\documentclass[11pt,reqno]{amsart}
\usepackage{amsmath,amssymb,amsthm}
\usepackage{mathtools}
\usepackage{graphicx}
\usepackage[margin=1.15in]{geometry}
\usepackage[colorlinks=true,linkcolor=blue,citecolor=blue,urlcolor=blue]{hyperref}

\numberwithin{equation}{section}
\theoremstyle{plain}
\newtheorem{theorem}{Theorem}[section]
\newtheorem{proposition}[theorem]{Proposition}
\newtheorem{lemma}[theorem]{Lemma}
\newtheorem{corollary}[theorem]{Corollary}
\theoremstyle{definition}

\newtheorem{remark}[theorem]{Remark}

\newcommand{\R}{\mathbb{R}}
\newcommand{\Ho}{H^{\circ}}
\newcommand{\mub}{\bar\mu}
\newcommand{\eps}{\varepsilon}
\newcommand{\gs}{\gamma_{*}}
\DeclareMathOperator{\dv}{div}
\DeclareMathOperator{\sgn}{sgn}

\begin{document}
\title[Asymptotics for Hardy--Emden--Fowler equations]
{Radial solutions of quasilinear Hardy--Emden--Fowler equations:
exact decay rates, Delaunay solutions, and nonexistence}

% authors
\author{Phuong Le}
\address{Phuong Le$^{1,2}$ (ORCID: 0000-0003-4724-7118)\newline
	$^1$Faculty of Economic Mathematics, University of Economics and Law, Ho Chi Minh City, Vietnam; \newline
	$^2$Vietnam National University, Ho Chi Minh City, Vietnam}
\email{phuongl@uel.edu.vn}


\begin{abstract}
Let $N>p>1$, $s>-p$, $0<\mu<\mub:=\big(\frac{N-p}{p}\big)^{p}$ and $q>p-1$.
We determine the behaviour at infinity of \emph{every} positive solution of
the quasilinear Hardy--Emden--Fowler equation
\[
-\big(r^{N-1}|v'|^{p-2}v'\big)'=\mu\,r^{N-1-p}v^{p-1}+r^{N-1+s}v^{q},
\qquad r>R_{0}.
\]
Three rates compete: the self-similar rate $r^{-\gs}$ with
$\gs=\frac{p+s}{q-p+1}$, and the two Hardy rates $r^{-\nu_{\pm}}$, where
$\nu_{\pm}$ are the roots of $h(\nu):=\nu^{p-1}(N-p-(p-1)\nu)=\mu$. Writing
$p^{*}_{s}=\frac{p(N+s)}{N-p}$ and $q_{S}=p-1+\frac{p+s}{\nu_{+}}$ for the
Hardy--Sobolev and Serrin exponents, we obtain a complete picture in four
regimes. If $q>p^{*}_{s}-1$, every solution realises exactly one of the three
rates, with a logarithmic correction in the borderline case $h(\gs)=\mu$, and
we compute the leading constants. If $q=p^{*}_{s}-1$ the system becomes
integrable: besides the singular self-similar solution and a heteroclinic
``bubble'' there is a one-parameter family of Fowler--Delaunay solutions, for
which $r^{\gs}v(r)$ is a non-constant periodic function of $\ln r$. If
$q_{S}<q<p^{*}_{s}-1$ only two orbits survive, and if $q\le q_{S}$ there is no
positive solution at all in an exterior domain. This is the quasilinear,
weighted counterpart of the picture obtained by Franca and Sfecci for $p=2$.
We work with the slope variables $x=-rv'/v$, $Y=r^{p+s}v^{q-p+1}$, in which the
equation becomes a planar system that is real-analytic on the half-plane
$\{x>0\}$ for every $1<p<N$, so that no restriction of the type $p\le2$ is
needed; the analysis then rests on
a universal barrier $x<\nu_{+}$, on an explicit Dulac function which
degenerates into a first integral exactly at $q=p^{*}_{s}-1$, and on an
explicit Lyapunov function.
\end{abstract}

\maketitle

%%%%%%%%%%%%%%%%%%%%%%%%%%%%%%%%%%%%%%%%%%%%%%%%%%%%%%%%%%%%%%%%%%%%%%%%%%%%%%%
\section{Introduction and results}
%%%%%%%%%%%%%%%%%%%%%%%%%%%%%%%%%%%%%%%%%%%%%%%%%%%%%%%%%%%%%%%%%%%%%%%%%%%%%%%

\subsection{The problem}

Throughout this paper we fix
\begin{equation}\label{eq:standing}
N>p>1,\qquad s>-p,\qquad 0<\mu<\mub:=\Big(\frac{N-p}{p}\Big)^{p},
\qquad q>p-1,
\end{equation}
and we study the behaviour as
$r\to+\infty$ of positive solutions of the ordinary differential equation
\begin{equation}\label{eq:ode}
-\big(r^{N-1}|v'|^{p-2}v'\big)'
=\mu\,r^{N-1-p}v^{p-1}+r^{N-1+s}v^{q},
\qquad r>R_{0}>0 .
\end{equation}
By a \emph{solution} of \eqref{eq:ode} we always mean a function
$v\in C^{1}\big((R_{0},+\infty)\big)$, positive, such that
$r\mapsto r^{N-1}|v'|^{p-2}v'$ is of class $C^{1}$ and \eqref{eq:ode} holds
pointwise.

Equation \eqref{eq:ode} is the radial form, in the exterior domain
$\R^{N}\setminus \overline{B_{R_{0}}}$, of the Hardy--Lane--Emden equation
\begin{equation}\label{eq:pde}
-\Delta_{p}u-\frac{\mu}{|x|^{p}}|u|^{p-2}u=|x|^{s}|u|^{q-1}u
\qquad\text{in }\R^{N}\setminus \overline{B_{R_{0}}} ,
\end{equation}
in the sense that $u(x)=v(|x|)$ solves \eqref{eq:pde} if and only if $v$ solves
\eqref{eq:ode}; solutions defined on the whole of $(0,+\infty)$, such as those
of alternatives {\rm(II)} and {\rm(III)} of Theorem~\ref{thm:critical},
correspond to solutions
of \eqref{eq:pde} in the punctured space $\R^{N}\setminus\{0\}$. Equation
\eqref{eq:ode} is also the radial form of the Finsler counterpart of
\eqref{eq:pde}, in which $\Delta_{p}$ is
replaced by the anisotropic operator
$\Delta^{H}_{p}u=\dv\big(H^{p-1}(\nabla u)\nabla H(\nabla u)\big)$ and $|x|$ by
the dual Minkowski norm $\Ho(x)$; see \cite{EMSV2024}. The equation arises in
the study of non-Newtonian fluids, of Caffarelli--Kohn--Nirenberg type
inequalities \cite{CKN}, and of critical problems involving simultaneously the
Sobolev and the Hardy exponent \cite{AbdellaouiPeral,GarciaAzoreroPeral}.

Three exact decay rates are visible in \eqref{eq:ode}. Set
\begin{equation}\label{eq:hgamma}
h(\gamma):=\gamma^{p-1}\big(N-p-(p-1)\gamma\big),\qquad \gamma>0 .
\end{equation}
The function $h$ vanishes at $0$ and at $\frac{N-p}{p-1}$, is strictly
increasing on $\big(0,\frac{N-p}{p}\big)$, strictly decreasing afterwards, and
attains its maximum $\mub$ exactly at $\frac{N-p}{p}$. Consequently
$0<\mu<\mub$ guarantees that
\begin{equation}\label{eq:nupm}
h(\nu)=\mu\quad\text{has exactly two roots}\quad
0<\nu_{-}<\frac{N-p}{p}<\nu_{+}<\frac{N-p}{p-1},
\end{equation}
and $r^{-\nu_{\pm}}$ are the two homogeneous solutions of the pure Hardy
equation $-\big(r^{N-1}|v'|^{p-2}v'\big)'=\mu r^{N-1-p}v^{p-1}$. On the other
hand
\begin{equation}\label{eq:gstar}
\gs:=\frac{p+s}{q-p+1}
\end{equation}
is the self-similar exponent: $v(r)=A_{*}r^{-\gs}$ solves \eqref{eq:ode}
identically provided $h(\gs)>\mu$, with
$A_{*}=\big(h(\gs)-\mu\big)^{1/(q-p+1)}$. Writing
\begin{equation}\label{eq:pstar}
p^{*}_{s}:=\frac{p(N+s)}{N-p}
\end{equation}
for the critical Hardy--Sobolev exponent attached to the weight $|x|^{s}$
(so that $p^{*}_{0}=p^{*}$ is the Sobolev exponent), an elementary computation
gives
\begin{equation}\label{eq:qsuper}
q>p^{*}_{s}-1\iff \gs<\frac{N-p}{p} .
\end{equation}
Since $q\mapsto\gs(q)$ is a decreasing bijection from $(p-1,+\infty)$ onto
$(0,+\infty)$, the position of $\gs$ relative to the three special values
$\nu_{-}<\frac{N-p}{p}<\nu_{+}$ splits the range of $q$ into four regimes. We
therefore introduce, besides $p^{*}_{s}$, the \emph{Serrin exponent}
\begin{equation}\label{eq:qS}
q_{S}:=p-1+\frac{p+s}{\nu_{+}} ,
\end{equation}
characterised by $\gs(q_{S})=\nu_{+}$, together with its counterpart
\begin{equation}\label{eq:qbar}
\bar q:=p-1+\frac{p+s}{\nu_{-}} ,
\end{equation}
characterised by $\gs(\bar q)=\nu_{-}$. Since $q\mapsto\gs(q)$ is decreasing,
with $\gs(q)\to+\infty$ as $q\to(p-1)^{+}$, and since
$\nu_{-}<\frac{N-p}{p}=\gs(p^{*}_{s}-1)<\nu_{+}$ by \eqref{eq:nupm}, we have
\[
p-1<q_{S}<p^{*}_{s}-1<\bar q<+\infty ,
\]
and
\begin{equation}\label{eq:regimes}
\begin{array}{llll}
\text{(A) supercritical:} & q>p^{*}_{s}-1 &\iff& \gs<\tfrac{N-p}{p};\\[2pt]
\text{(B) critical:} & q=p^{*}_{s}-1 &\iff& \gs=\tfrac{N-p}{p};\\[2pt]
\text{(C) subcritical:} & q_{S}<q<p^{*}_{s}-1 &\iff& \tfrac{N-p}{p}<\gs<\nu_{+};\\[2pt]
\text{(D) sub-Serrin:} & p-1<q\le q_{S} &\iff& \gs\ge\nu_{+} .
\end{array}
\end{equation}
When $\mu\to0^{+}$ and $s=0$ one has $\nu_{+}\to\frac{N-p}{p-1}$ and
$q_{S}\to\frac{N(p-1)}{N-p}$, the classical Serrin exponent for nonexistence of
positive solutions in exterior domains; for $p=2$ this is $\frac{N}{N-2}$.

Since $h$ increases on $\big(0,\frac{N-p}{p}\big)$, decreases afterwards, and
equals $\mu$ exactly at $\nu_{\pm}$,
\begin{equation}\label{eq:trichotomycases}
h(\gs)>\mu\iff\nu_{-}<\gs<\nu_{+},
\qquad
h(\gs)\le\mu\iff \gs\le\nu_{-}\ \text{ or }\ \gs\ge\nu_{+} ,
\end{equation}
so that the interior equilibrium $P_{*}$ of \eqref{eq:system} below exists
exactly when $\nu_{-}<\gs<\nu_{+}$: always in regimes {\rm(B)} and {\rm(C)},
and in regime {\rm(A)} precisely when $\gs>\nu_{-}$.

\subsection{Related literature}\label{ss:lit}

The Fowler transformation and the associated phase-plane analysis form a
well-established body of work, and it is against this background that our
contribution should be measured. We describe it in four groups.

\smallskip
\noindent\emph{(a) The homogeneous Hardy equation.} When the nonlinearity is
absent, \eqref{eq:ode} reduces to the Euler--Lagrange equation of the (weighted)
Hardy inequality, and its radial solutions behave exactly like $r^{-\nu_{-}}$ or
$r^{-\nu_{+}}$. A complete description, in the more general weighted form
$\dv\big(|x|^{\alpha}|\nabla u|^{p-2}\nabla u\big)
+\mu|x|^{\alpha-p}|u|^{p-2}u=0$ and for all three positions of $\mu$ relative
to the corresponding Hardy constant $\big|\frac{N-p+\alpha}{p}\big|^{p}$
(which is $\mub$ for $\alpha=0$), was given by Itakura and Tanaka
\cite{ItakuraTanaka}; see also Bakhadda and Bouzelmate \cite{BakhaddaBouzelmate}
for a treatment by a Fowler-type transformation when $2<p<N$. These works supply
the two Hardy rates $\nu_{\pm}$, but say nothing about the interaction with a
nonlinear term.

\smallskip
\noindent\emph{(b) The $p$-Laplacian without Hardy potential.} For $\mu=0$ the
asymptotic classification of positive radial solutions of \eqref{eq:ode} is
classical: it goes back to Guedda and V\'eron \cite{GueddaVeron}, who already
identified the two thresholds that, for $p=2$, $s=0$, are the Serrin exponent
$\frac{N}{N-2}$ and the Sobolev exponent $\frac{N+2}{N-2}$, and to
Bidaut-V\'eron \cite{BidautVeron1989,BV-V}, who introduced the Fowler
transformation
\begin{equation}\label{eq:fowlervars}
\mathsf X=u(r)\,r^{\frac{N-p}{p}},\qquad
\mathsf Y=u'(r)|u'(r)|^{p-2}r^{\frac{N(p-1)}{p}} ,
\end{equation}
in the quasilinear setting. The resulting planar system was studied
systematically by Franca and collaborators
\cite{Franca2004,Franca2008,Franca2009,FrancaJohnson} and, in the semilinear
case $p=2$, by Johnson--Pan--Yi \cite{JPY} and
Bianchi--Egnell \cite{BianchiEgnell}, where the whole theory is by now
classical; see \cite{QS} for a systematic account.
In the generalised form of \cite{Franca2008}, where $\Delta_{p}u+f(u,|x|)=0$ is
treated for a wide class of nonlinearities, the transformation reads
$x_{l}=u(r)r^{\alpha_{l}}$, $y_{l}=u'(r)|u'(r)|^{p-2}r^{\beta_{l}}$ with
$\alpha_{l}=\frac{p}{l-p}$, $\beta_{l}=\frac{(p-1)l}{l-p}$, and yields
\begin{equation}\label{eq:francasystem}
\dot x_{l}=\alpha_{l}x_{l}+y_{l}|y_{l}|^{\frac{2-p}{p-1}},
\qquad
\dot y_{l}=\gamma_{l}y_{l}-g_{l}(x_{l},t),
\qquad \gamma_{l}:=\beta_{l}-N+1 ,
\end{equation}
$g_{l}$ denoting the nonlinearity $f$ read in the new variables.
In these variables the phase portrait of the autonomous system is completely
understood. The structure of the positive solutions changes at the two
thresholds $\sigma_{S}=\frac{p(N-1)}{N-p}$ (the Serrin exponent, in the
normalisation in which the nonlinearity is written $u^{\sigma-1}$, i.e.\
$\sigma=q+1$) and
$p^{*}=\frac{Np}{N-p}$, giving the trichotomy called \emph{Sub}, \emph{Crit},
\emph{Sup} in \cite[\S1]{Franca2008}: the interior equilibrium is unstable for
$\sigma_{S}<q+1<p^{*}$, a \emph{centre} at $q+1=p^{*}$ and stable for
$q+1>p^{*}$.
At the critical exponent the system is moreover Hamiltonian, with conserved
energy
$\mathcal E=\frac{N-p}{p}\mathsf X\mathsf Y
+\frac{p-1}{p}|\mathsf Y|^{\frac{p}{p-1}}
+\frac{|\mathsf X|^{p^{*}}}{p^{*}}$,
so that the equilibrium is surrounded by periodic orbits and enclosed by a
homoclinic loop; see \cite[Fig.~1]{FrancaSfecci} and \cite[\S2]{DFS2026}. A direct
computation gives $p^{*}=p^{*}_{s}$ when $s=0$, and $\sigma_{S}=q_{S}+1$ in the
limit $\mu\to0^{+}$, so the four regimes \eqref{eq:regimes} reduce exactly to this
trichotomy, and our critical-regime result (Theorem~\ref{thm:critical}) is the
analogue of this classical picture in the presence of a Hardy potential and of
the weight $|x|^{s}$. The classification of \cite{Franca2004} for
$\Delta_{p}u+K(|x|)u|u|^{q-1}=0$, $q>\frac{N(p-1)}{N-p}$, is the limiting case
$\mu=0$ of the range considered here.

\smallskip
\noindent\emph{(c) The Laplacian with a Hardy potential.} The case $p=2$,
$\mu>0$ has been treated in depth by Franca and Sfecci \cite{FrancaSfecci},
who study $\Delta u+\frac{\mathsf h(|x|)}{|x|^{2}}u+f(u,|x|)=0$ with
$\mathsf h$ not necessarily constant, again by the Fowler transformation. Their
analysis contains, for $p=2$ and $\mathsf h\equiv\eta$ constant, precisely the
structure we establish here for general $p$: with
$\varkappa(\eta)=\frac{(N-2)-\sqrt{(N-2)^{2}-4\eta}}{2}$ they introduce the two
exponents
\[
2_{*}(\eta)=2\,\frac{N+\sqrt{(N-2)^{2}-4\eta}}{N-2+\sqrt{(N-2)^{2}-4\eta}},
\qquad
\mathrm I(\eta)=2\,\frac{N-\sqrt{(N-2)^{2}-4\eta}}{N-2-\sqrt{(N-2)^{2}-4\eta}} ,
\]
and prove that the interior equilibrium is unstable if
$2_{*}(\eta)<l<2^{*}$, a centre if $l=2^{*}$, and stable if
$2^{*}<l<\mathrm I(\eta)$, the parameter $l$ playing, for $s=0$, the role of
our $q+1$. A direct computation shows that, for $p=2$, $s=0$ and $\eta=\mu$,
\begin{equation}\label{eq:matchFS}
\varkappa(\eta)=\nu_{-},\qquad
N-2-\varkappa(\eta)=\nu_{+},\qquad
2_{*}(\eta)=q_{S}+1,\qquad
\mathrm I(\eta)=\bar q+1 ,
\end{equation}
with $\bar q$ as in \eqref{eq:qbar}; so the exponents \eqref{eq:qS},
\eqref{eq:qbar} and the four regimes \eqref{eq:regimes} are the quasilinear,
weighted counterparts of theirs. See also \cite{Bae,CirsteaFarcaseanu,WeiDu,FengTanWei},
\cite{JeongLee2013} for stable and finite Morse index solutions,
\cite{HeXiang2016a} for the effect of an additional mass term, and
\cite{AFP,Xiang2015,Xiang2017,OSV,HeXiang2016b} for the critical
Hardy--Sobolev exponent.

\smallskip
\noindent\emph{(d) The gap.} The two extensions above have not been combined:
\cite{FrancaSfecci} treats Hardy potentials only for $p=2$, as is visible
from their system, in which the Hardy term enters the second equation
\emph{linearly}, through the constant matrix
$\begin{psmallmatrix}\alpha&1\\-\mathsf h&\gamma\end{psmallmatrix}$, a feature
available only when $p=2$, while
\cite{Franca2004,Franca2008,Franca2009,DFS2026} treat general $p$ but without
Hardy term. Moreover, even in the latter works the range of $p$ is restricted:
the term $y_{l}|y_{l}|^{\frac{2-p}{p-1}}$ in \eqref{eq:francasystem} is only
H\"older continuous on $\{y_{l}=0\}$ when $p>2$, whence the standing
assumptions $1<p\le2$ in \cite[\S2]{Franca2008} and
$\frac{2N}{2+N}\le p\le2$ in \cite[Rem.~7]{DFS2026}. The slope variables used
below remove this restriction altogether. To the best of our knowledge the
quasilinear equation with a Hardy potential, $\mu>0$ and $p\ne2$, has been
analysed only at the critical
exponent $q=p^{*}_{s}-1$ and for finite-energy solutions
\cite{AFP,Xiang2015,Xiang2017,OSV}; the survey in the introduction of
\cite{BidautVeronChen} confirms this. Filling this gap for the whole range
$q>p-1$ is the purpose of the present paper. We note that the equation with an
\emph{absorption} term,
$-\Delta_{p}u+\mu\frac{u^{p-1}}{|x|^{p}}+|x|^{\theta}u^{q}=0$, has just received
a complete treatment by Bidaut-V\'eron and Chen \cite{BidautVeronChen}. There
$\mu$ ranges over all of $\R$, and our Hardy coefficient corresponds to their
$-\mu$; up to that convention, the two equations differ precisely by the sign in
front of the power term. That change of sign alters the dynamics substantially,
and our results are not a corollary of theirs.

\subsection{Method, and what is new}

The obstruction to transferring the results of \cite{FrancaSfecci} to $p\ne2$ is
not a matter of technique but of structure, and it is twofold. First, in the
Fowler variables \eqref{eq:fowlervars} the Hardy term enters the second equation
linearly when $p=2$, so that the linear part of the system is a genuine
$2\times2$ matrix whose eigenvalues, exponential dichotomies and invariant
manifolds can be computed; for $p\ne2$ this linearity is destroyed. Second, in
those same variables the first equation contains
$y_{l}|y_{l}|^{\frac{2-p}{p-1}}$, which for $p>2$ is only H\"older continuous
where $y_{l}=0$; this is precisely why the quasilinear results of
\cite{Franca2008,DFS2026} carry the restriction $p\le2$. The customary
alternative
reduction, $w(t)=r^{\gs}v(r)$, $t=\ln r$, leads to
\[
\big(\Phi(w'-\gs w)\big)'+\big(N-p-\gs(p-1)\big)\Phi(w'-\gs w)
+\mu|w|^{p-2}w+|w|^{q-1}w=0,\qquad \Phi(z)=|z|^{p-2}z ,
\]
and writing this as a planar system again requires inverting $\Phi$, with the
same loss of regularity: the standard stable manifold and centre manifold
theorems do not apply when $p>2$.

We therefore work with the \emph{slope variable} $x=-rv'/v$, in the spirit of
Bidaut-V\'eron and Giacomini \cite{BVG} and of \cite{BidautVeronChen}, and set
\begin{equation}\label{eq:xY}
x(t):=-\frac{r\,v'(r)}{v(r)},\qquad
Y(t):=r^{p+s}v(r)^{q-p+1},\qquad t=\ln r ,
\end{equation}
so that \eqref{eq:ode} becomes (Lemma~\ref{lem:system})
\begin{equation}\label{eq:system}
\dot x=\frac{x^{2-p}}{p-1}\big[\mu+Y-h(x)\big],
\qquad
\dot Y=(q-p+1)\,Y\,\big[\gs-x\big] .
\end{equation}
We write $\mathbf F=(F,G)$ for the associated vector field, i.e.
\begin{equation}\label{eq:FG}
F(x,Y)=\frac{x^{2-p}}{p-1}\big[\mu+Y-h(x)\big],
\qquad
G(x,Y)=(q-p+1)\,Y\,(\gs-x).
\end{equation}
System \eqref{eq:system} is real-analytic on the open half-plane
$(0,+\infty)\times\R$ for \emph{every} $p>1$: the only possible singularity is the
factor $x^{2-p}$, which is analytic for $x>0$. In particular no restriction of
the type $p\le2$ is needed here, in contrast with
\cite[\S2]{Franca2008} and \cite[Rem.~7]{DFS2026}; the change of unknown from
$y_{l}$ to the slope $x$ is what removes it. All our results therefore hold on
the whole admissible range $1<p<N$, covering the singular case $p<2$ and the
degenerate case $p>2$ on an equal footing.

The price to pay is that the slope variables are adapted to \emph{positive,
decreasing} solutions, since $x=-rv'/v$ is defined only where $v>0$; this is
harmless for the questions studied here, since Theorem~\ref{thm:apriori} shows
that the monotonicity is automatic. (It is the source-case analogue
of system (4.8) of \cite{BidautVeronChen}, written there in terms of
$x^{p-1}$.) All the classical local theory is therefore available without any
restriction on $p$.

The equilibria of \eqref{eq:system} in $\{x>0,\ Y\ge0\}$ are
\begin{equation}\label{eq:equilibria}
P_{*}:=\big(\gs,\ h(\gs)-\mu\big)\ \ (\text{only if } h(\gs)>\mu),
\qquad
P_{-}:=(\nu_{-},0),\qquad P_{+}:=(\nu_{+},0),
\end{equation}
and they correspond exactly to the three decay rates $\gs$, $\nu_{-}$,
$\nu_{+}$. Figure~\ref{fig:portraits} shows the resulting phase portraits in
the four regimes \eqref{eq:regimes}.

\begin{figure}[ht]
\centering
\includegraphics[width=\textwidth]{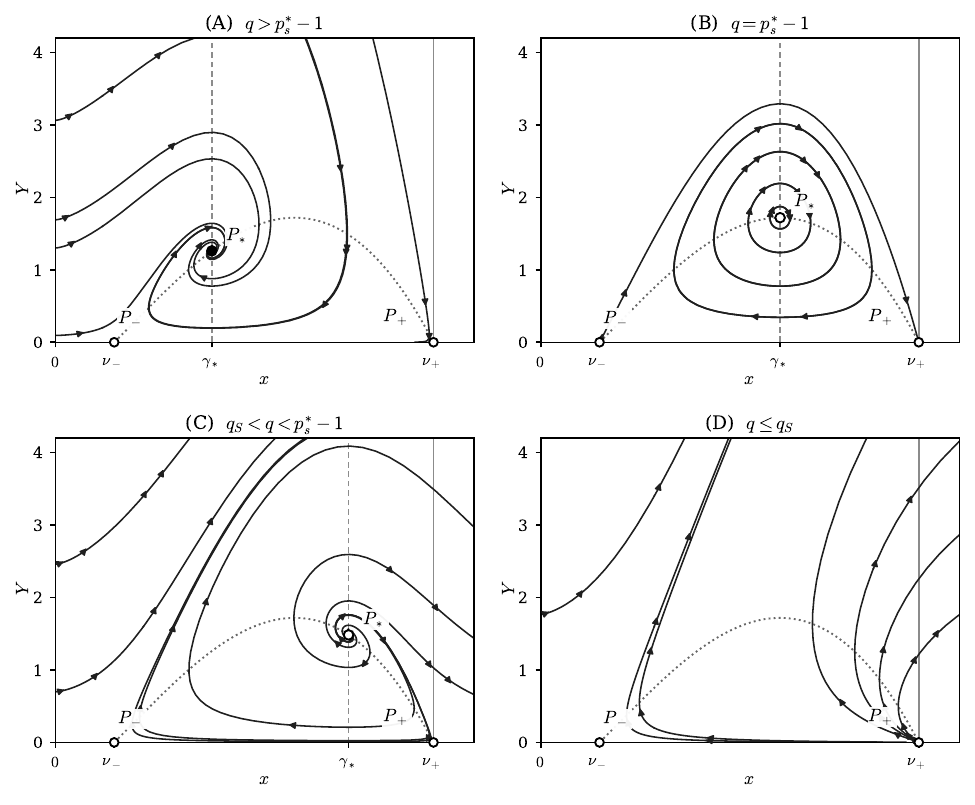}
\caption{Phase portraits of \eqref{eq:system} in the quadrant
$\{x>0,\ Y>0\}$, for $N=6$, $p=2.5$, $s=0.7$, $\mu=0.6$ and, respectively,
$q=5$, $q=p^{*}_{s}-1$, $q=\tfrac12(q_{S}+p^{*}_{s}-1)$ and
$q=p-1+0.55\,(q_{S}-p+1)$. Dotted: the nullcline $Y=h(x)-\mu$; dashed: the
nullcline $x=\gs$; the solid vertical line is the barrier $x=\nu_{+}$ of
Theorem~\ref{thm:apriori}. In {\rm(A)} the interior equilibrium $P_{*}$ is
asymptotically stable and $P_{+}$ is a saddle; in {\rm(B)} $P_{*}$ is a centre,
surrounded by the Fowler--Delaunay periodic orbits of
Theorem~\ref{thm:critical}{\rm(II)} and enclosed by the heteroclinic orbit
$\Gamma_{0}$ joining $P_{-}$ to $P_{+}$; in {\rm(C)} $P_{*}$ is a repeller and
the only orbit converging to an equilibrium is the stable manifold of $P_{+}$;
in {\rm(D)} one has $\gs>\nu_{+}$, $P_{*}$ has disappeared, and every orbit
crosses $x=\nu_{+}$, so that the corresponding solution $v$ vanishes at a
finite radius. The orbits were obtained by numerical integration of
\eqref{eq:system}.}
\label{fig:portraits}
\end{figure}

What replaces the linear-algebraic machinery of \cite{FrancaSfecci} is the
following. First, a universal a priori barrier $x<\nu_{+}$
(Theorem~\ref{thm:apriori}), whose violation forces $v$ to vanish at a finite
radius. Second, and this is the main new tool, the explicit function
\eqref{eq:B}, which is a Dulac function when $q\ne p^{*}_{s}-1$ and an
integrating factor when $q=p^{*}_{s}-1$: its divergence
\eqref{eq:divB} keeps a strict sign on the whole quadrant in the former case,
so that periodic orbits and separatrix cycles are excluded and
Poincar\'e--Bendixson applies, and vanishes identically in the latter, making
the system integrable. It is worth stressing that the exponent at which this
divergence changes sign is exactly $p^{*}_{s}-1$; the single function $B$ thus
encodes, in the quasilinear weighted setting, the trichotomy
unstable / centre / stable found in \cite[\S2.1]{FrancaSfecci} for $p=2$.
Third, an explicit Lyapunov function (Theorem~\ref{thm:lyap}) which, unlike the
qualitative invariant-manifold arguments, produces a \emph{quantitative} basin
of attraction for the self-similar profile. All the leading constants are
computed explicitly.

\subsection{Main results}

Our first result combines two facts: every positive solution in an exterior
domain is eventually decreasing and tends to $0$, and its slope obeys a
universal a priori bound. The latter is the exact analogue, in the present
setting, of the elementary fact that a positive solution cannot decay faster
than any power without vanishing.

\begin{theorem}\label{thm:apriori}
Assume \eqref{eq:standing} and let $v>0$ solve \eqref{eq:ode} on
$(R_{0},+\infty)$. Then
\begin{equation}\label{eq:auto}
v'(r)<0\ \text{ for all large }r,
\qquad
\lim_{r\to+\infty}v(r)=0 ,
\end{equation}
and, $(x,Y)$ being as in \eqref{eq:xY}, there is $t_{0}\in\R$ such that
\begin{equation}\label{eq:apriori}
0<x(t)<\nu_{+}\ \text{ for every }t\ge t_{0},
\qquad \liminf_{t\to\infty}x(t)>0,
\qquad \sup_{t\ge t_{0}}Y(t)<+\infty .
\end{equation}
\end{theorem}

The strict bound $x<\nu_{+}$ is sharp: it is attained in the limit along the
fast-decay branch. It says that $\nu_{+}$ is a genuine barrier, crossing which
forces $v$ to vanish at a finite radius (Remark~\ref{rem:barrier}). Note that \eqref{eq:auto} uses only
the presence of the source terms on the right-hand side of \eqref{eq:ode}; it
lets us dispense with any decay assumption, so that all the results below are
stated for \emph{arbitrary} positive solutions in an exterior domain.

The second result is the key structural fact. Recall that a positive $C^{1}$
function $B$ is a \emph{Dulac function} for a vector field $\mathbf F$ on an
open set $U\subset\R^{2}$ if $\dv(B\mathbf F)$ has a strict sign on $U$, and an
\emph{integrating factor} if $\dv(B\mathbf F)\equiv0$ on $U$. The following
single function $B$ plays both roles, according to the regime.

\begin{theorem}\label{thm:dulac}
Assume \eqref{eq:standing} and set
\begin{equation}\label{eq:B}
B(x,Y):=x^{p-2}\,Y^{\,b},\qquad b:=\frac{N-p}{p+s}-1\ (>-1).
\end{equation}
Then, denoting by $\mathbf F$ the vector field of \eqref{eq:system},
\begin{equation}\label{eq:divB}
\dv\big(B\mathbf F\big)(x,Y)
=\Big(p-\frac{N-p}{\gs}\Big)\,x^{p-1}\,Y^{\,b}
\qquad\text{for all }x>0,\ Y>0 ,
\end{equation}
whose sign is negative in regime {\rm(A)}, zero in regime {\rm(B)} and positive
in regimes {\rm(C)}--{\rm(D)}. Consequently:
\begin{enumerate}
\item[\rm(i)] if $q\ne p^{*}_{s}-1$, system \eqref{eq:system} admits neither a
periodic orbit nor a separatrix cycle in $\{x>0,\ Y\ge0\}$, and every solution
with $x>0$, $Y>0$ satisfying \eqref{eq:apriori} converges, as $t\to+\infty$, to
one of the equilibria \eqref{eq:equilibria};
\item[\rm(ii)] if $q=p^{*}_{s}-1$, the field $B\mathbf F$ is Hamiltonian on
$\{x>0,\ Y>0\}$: setting
\begin{equation}\label{eq:firstintegral}
\mathcal H(x,Y):=\frac{Y^{\,b+1}}{p-1}
\left[\frac{\mu-h(x)}{b+1}+\frac{Y}{b+2}\right],
\qquad b+1=\frac{p}{q-p+1},
\end{equation}
one has $\partial_{Y}\mathcal H=BF$ and $\partial_{x}\mathcal H=-BG$, so
$\mathcal H$ is a first integral of \eqref{eq:system}.
\end{enumerate}
\end{theorem}

Combining Theorems~\ref{thm:apriori} and \ref{thm:dulac} with the local
analysis at each equilibrium yields the exact decay rates. This is our main
result.

\begin{theorem}[Supercritical regime]\label{thm:main}
Assume \eqref{eq:standing}, $q>p^{*}_{s}-1$, and let $v>0$ solve \eqref{eq:ode}
on $(R_{0},+\infty)$. Then exactly one of the following holds.
\begin{enumerate}
\item[\rm(I)] \emph{(Self-similar decay; possible only if $\gs>\nu_{-}$, i.e.\
$q<\bar q$.)}
\[
\lim_{r\to+\infty}r^{\gs}v(r)=\big(h(\gs)-\mu\big)^{\frac{1}{q-p+1}},
\qquad
\lim_{r\to+\infty}\frac{r\,v'(r)}{v(r)}=-\gs .
\]
\item[\rm(II)] \emph{(Slow Hardy decay; possible only if $\gs<\nu_{-}$, i.e.\
$q>\bar q$.)}
\[
\lim_{r\to+\infty}r^{\nu_{-}}v(r)=C\in(0,+\infty),
\qquad
\lim_{r\to+\infty}\frac{r\,v'(r)}{v(r)}=-\nu_{-} .
\]
\item[\rm(III)] \emph{(Fast Hardy decay.)}
\[
\lim_{r\to+\infty}r^{\nu_{+}}v(r)=C\in(0,+\infty),
\qquad
\lim_{r\to+\infty}\frac{r\,v'(r)}{v(r)}=-\nu_{+} .
\]
\item[\rm(IV)] \emph{(Degenerate case; possible only if $\gs=\nu_{-}$, i.e.\
$q=\bar q$.)}
\[
v(r)=C_{*}\,r^{-\gs}(\ln r)^{-\frac{1}{q-p+1}}\big(1+o(1)\big),
\]
\[
\text{where}\qquad
C_{*}=\left(\frac{(p-1)\big[(q-p+1)(N-p)-p(p+s)\big](p+s)^{p-2}}
{(q-p+1)^{p}}\right)^{\frac{1}{q-p+1}},
\]
and moreover
$\dfrac{r v'(r)}{v(r)}=-\gs-\dfrac{1}{(q-p+1)\ln r}+O\big((\ln r)^{-2}\big)$.
\end{enumerate}
In cases {\rm(I)}, {\rm(II)}, {\rm(III)} the convergence of $rv'/v$ is
exponential in $t=\ln r$. All four alternatives occur, for suitable values of
the parameters.
\end{theorem}

In the range $\gs>\nu_{-}$, which is the generic one, the four alternatives of
Theorem~\ref{thm:main} reduce to the two rates $\gs$ and $\nu_{+}$, and these
are separated by the value of $\lim_{r\to\infty} r^{\gs}v(r)$:

\begin{corollary}\label{cor:H}
Assume \eqref{eq:standing} and $p^{*}_{s}-1<q<\bar q$ (equivalently
$\gs>\nu_{-}$ in regime {\rm(A)}), and let $v$ be as in
Theorem~\ref{thm:main}. Then
\[
\lim_{r\to+\infty}r^{\gs}v(r)=0
\iff
\lim_{r\to+\infty}r^{\nu_{+}}v(r)\in(0,+\infty) ,
\]
and in that case
\[
\int^{\infty}\!r^{N-1}|v'|^{p}\,dr<\infty,\quad
\int^{\infty}\!r^{N-1-p}v^{p}\,dr<\infty,\quad
\int^{\infty}\!r^{N-1+s}v^{q+1}\,dr<\infty .
\]
\end{corollary}

In the remaining regimes the picture changes completely, and this is the second
half of our results. In the critical regime the system becomes integrable and a
one-parameter family of \emph{Delaunay-type} solutions appears; in the
subcritical regime the self-similar profile turns from an attractor into a
repeller, so that only two orbits survive; and below the Serrin exponent no
solution survives at all.

\begin{theorem}[Critical regime]\label{thm:critical}
Assume \eqref{eq:standing} and $q=p^{*}_{s}-1$, so that $\gs=\frac{N-p}{p}$ and
$Y_{*}=\mub-\mu>0$. Let $v>0$ solve \eqref{eq:ode} on $(R_{0},+\infty)$ and let
$\mathcal H$ be as in \eqref{eq:firstintegral}. Then
$\mathcal H$ is constant along the orbit, its value $\kappa$ satisfies
$\mathcal H(P_{*})\le\kappa\le0$, and exactly one of the following holds.
\begin{enumerate}
\item[\rm(I)] $\kappa=\mathcal H(P_{*})$ and $v(r)=A_{*}\,r^{-\gs}$ for every
$r>R_{0}$, with $A_{*}=(\mub-\mu)^{1/(q-p+1)}$.
\item[\rm(II)] $\mathcal H(P_{*})<\kappa<0$ and the orbit is a periodic orbit
of \eqref{eq:system} surrounding $P_{*}$. In this case $v$ extends to a
positive solution of \eqref{eq:ode} on all of $(0,+\infty)$ and
\[
r^{\gs}v(r)=\Psi(\ln r),\qquad r>0,
\]
where $\Psi$ is a positive, non-constant, $T$-periodic function for some
$T=T(\kappa)>0$. In particular $\limsup_{r\to\infty}r^{\gs}v(r)$ and
$\liminf_{r\to\infty}r^{\gs}v(r)$ are distinct positive numbers, so $v$ has no
exact decay rate.
\item[\rm(III)] $\kappa=0$ and the orbit is the heteroclinic connection from
$P_{-}$ to $P_{+}$ lying on the curve
$Y=\frac{N+s}{N-p}\big(h(x)-\mu\big)$, $x\in(\nu_{-},\nu_{+})$. In this case
$v$ extends to a positive solution of \eqref{eq:ode} on all of $(0,+\infty)$,
and the extension satisfies
\[
\lim_{r\to+\infty}r^{\nu_{+}}v(r)\in(0,+\infty),
\qquad
\lim_{r\to0^{+}}r^{\nu_{-}}v(r)\in(0,+\infty) .
\]
\end{enumerate}
All three alternatives occur, and the orbits arising in {\rm(II)} form a
one-parameter family, indexed by the value
$\kappa\in\big(\mathcal H(P_{*}),0\big)$.
\end{theorem}

\begin{theorem}[Subcritical regime]\label{thm:sub}
Assume \eqref{eq:standing} and $q_{S}<q<p^{*}_{s}-1$, and let $v>0$ solve
\eqref{eq:ode} on $(R_{0},+\infty)$. Then exactly one of the following holds,
the labels being those of Theorem~\ref{thm:main}.
\begin{enumerate}
\item[\rm(I)] $v(r)=A_{*}\,r^{-\gs}$ for every $r>R_{0}$, with
$A_{*}=\big(h(\gs)-\mu\big)^{1/(q-p+1)}$;
\item[\rm(III)] $\displaystyle\lim_{r\to+\infty}r^{\nu_{+}}v(r)=C\in(0,+\infty)$
and $\displaystyle\lim_{r\to+\infty}\frac{rv'(r)}{v(r)}=-\nu_{+}$.
\end{enumerate}
Both occur; the solutions of type {\rm(III)} correspond to the single orbit
given by the branch of the one-dimensional stable manifold of $P_{+}$ contained
in $\{Y>0\}$, and therefore form a one-parameter family, the parameter being
the scaling $v\mapsto\lambda^{\gs}v(\lambda\,\cdot\,)$, $\lambda>0$, under
which \eqref{eq:ode} is invariant.
\end{theorem}

\begin{theorem}[Nonexistence below the Serrin exponent]\label{thm:nonex}
Assume \eqref{eq:standing} and $p-1<q\le q_{S}$. Then \eqref{eq:ode} has
\emph{no} positive solution on any exterior interval $(R_{0},+\infty)$.
\end{theorem}

\begin{remark}
For $\mu\to0^{+}$ and $s=0$ the threshold $q_{S}$ tends to
$\frac{N(p-1)}{N-p}$, so Theorem~\ref{thm:nonex} recovers, in the limit, the
classical Serrin nonexistence exponent in exterior domains, identified for
$\mu=0$ in \cite{GueddaVeron}; for $p=2$ this is $\frac{N}{N-2}$. For $p=2$ and
$s=0$ but $\mu>0$, $q_{S}+1$ coincides with the exponent $2_{*}(\eta)$ of
\cite{FrancaSfecci}, see \eqref{eq:matchFS}. Since $\mu>0$ strictly increases
$q_{S}$ (because $\nu_{+}$ decreases with $\mu$), a positive Hardy coefficient
\emph{enlarges} the nonexistence range, as one expects from the fact that the
Hardy term then acts as an additional source.
\end{remark}

\begin{remark}\label{rem:delaunay}
The solutions of type {\rm(II)} in Theorem~\ref{thm:critical} are the analogue,
in the present setting, of the classical Fowler--Delaunay solutions of the
critical Yamabe equation $-\Delta u=u^{(N+2)/(N-2)}$, for which
$|x|^{(N-2)/2}u$ is periodic in $\ln|x|$; alternative {\rm(III)} is the analogue
of the Aubin--Talenti bubble, with the difference that, for $\mu>0$, it is
\emph{not} bounded at the origin but blows up like $r^{-\nu_{-}}$, regular
solutions being absent as soon as the Hardy term is present (compare
\cite[Rem.~2.7]{FrancaSfecci}), and {\rm(I)} of the singular solution. That the
interior equilibrium is a centre at the critical exponent, hence surrounded by
periodic orbits and enclosed by a loop joining the equilibria on $\{Y=0\}$, is
classical for $\mu=0$ \cite{Franca2004,Franca2009,JPY,BianchiEgnell} and was
established for $p=2$ with a Hardy potential in \cite[\S2.1]{FrancaSfecci}. In
the Fowler variables \eqref{eq:fowlervars} that loop is \emph{homoclinic} to
the origin; in the slope variables \eqref{eq:xY} the origin is blown up into
the two points $P_{\pm}$, and the loop becomes the \emph{heteroclinic} orbit
$\Gamma_{0}$ of Theorem~\ref{thm:critical}{\rm(III)}. Theorem~\ref{thm:critical}
extends this picture to $p\ne2$ with a Hardy potential and a weight $|x|^{s}$,
and makes it quantitative: the first integral \eqref{eq:firstintegral} is
explicit, and the value $\kappa$ of $\mathcal H$ parametrises the family.
\end{remark}

Our last result complements Theorem~\ref{thm:main} quantitatively: it exhibits
an explicit region of initial data leading to the self-similar profile
{\rm(I)}. Assuming $h(\gs)>\mu$, define, for $x>0$ and $Y>0$,
\begin{equation}\label{eq:Lyapunov}
V(x,Y):=(q-p+1)(p-1)\left[\frac{x^{p}}{p}-\gs\frac{x^{p-1}}{p-1}\right]
+\big(Y-Y_{*}\ln Y\big),
\qquad Y_{*}:=h(\gs)-\mu .
\end{equation}

\begin{theorem}\label{thm:lyap}
Assume \eqref{eq:standing} and $p^{*}_{s}-1<q<\bar q$, so that
$h(\gs)>\mu$. Then, along any solution of \eqref{eq:system} with $x>0$,
$Y>0$, the function $V$ of \eqref{eq:Lyapunov} satisfies
\begin{equation}\label{eq:Vdot}
\frac{d}{dt}V\big(x(t),Y(t)\big)
=-(q-p+1)\,\big(x-\gs\big)\big(h(x)-h(\gs)\big).
\end{equation}
Let $\gs^{\sharp}$ be the unique number in $\big(\frac{N-p}{p},
\frac{N-p}{p-1}\big)$ with $h(\gs^{\sharp})=h(\gs)$. Then $V$ is a Lyapunov
function on the strip $\{0<x<\gs^{\sharp},\ Y>0\}$: one has $\dot V\le0$
there, with equality exactly on the segment $\{x=\gs\}$. Setting
\begin{equation}\label{eq:basin}
c_{0}:=V\big(\gs^{\sharp},Y_{*}\big),
\qquad
\mathcal W:=\text{the component of }\big\{V<c_{0}\big\}\text{ containing }P_{*},
\end{equation}
the set $\mathcal W$ is an open bounded forward-invariant neighbourhood of
$P_{*}$ contained in $\{0<x<\gs^{\sharp}\}$, and every solution $v$ of
\eqref{eq:ode} whose orbit enters $\mathcal W$ satisfies alternative {\rm(I)}
of Theorem~\ref{thm:main}.
\end{theorem}

The paper is organised as follows. Section~\ref{sec:system} derives
\eqref{eq:system} and records the linearisations at the equilibria.
Section~\ref{sec:apriori} proves Theorem~\ref{thm:apriori}.
Section~\ref{sec:dulac} proves Theorem~\ref{thm:dulac}.
Section~\ref{sec:lyap} proves Theorem~\ref{thm:lyap}.
Section~\ref{sec:main} proves Theorem~\ref{thm:main} and
Corollary~\ref{cor:H}. Finally Section~\ref{sec:critical} treats the critical
regime and Section~\ref{sec:sub} the subcritical and sub-Serrin ones.

%%%%%%%%%%%%%%%%%%%%%%%%%%%%%%%%%%%%%%%%%%%%%%%%%%%%%%%%%%%%%%%%%%%%%%%%%%%%%%%
\section{The Emden--Fowler system}\label{sec:system}
%%%%%%%%%%%%%%%%%%%%%%%%%%%%%%%%%%%%%%%%%%%%%%%%%%%%%%%%%%%%%%%%%%%%%%%%%%%%%%%

\begin{lemma}\label{lem:system}
Let $v>0$ solve \eqref{eq:ode} on an interval $I\subset(0,\infty)$ with
$v'<0$ on $I$. Then the functions $x,Y$ defined in \eqref{eq:xY} are positive,
of class $C^{1}$ in $t=\ln r$, and satisfy \eqref{eq:system}. Conversely, every
solution $(x,Y)$ of \eqref{eq:system} with $x>0$, $Y>0$ on an interval
$J\subset\R$ arises in this way from a solution $v>0$ of \eqref{eq:ode}, with
$v'<0$, on $\{\mathrm e^{t}:t\in J\}$; $v$ is recovered from
\begin{equation}\label{eq:recover}
v(r)=\big(Y(\ln r)\big)^{\frac{1}{q-p+1}}\,r^{-\gs},
\qquad
\frac{r v'(r)}{v(r)}=-x(\ln r) .
\end{equation}
\end{lemma}

\begin{proof}
We first note that $v'\in C^{1}(I)$, so that $x$ is indeed of class $C^{1}$:
the function $(-v')^{p-1}=-r^{1-N}\big(r^{N-1}|v'|^{p-2}v'\big)$ is $C^{1}$ and
\emph{positive} on $I$, and $z\mapsto z^{\frac{1}{p-1}}$ is smooth on
$(0,+\infty)$. This is where the slope variable pays off: the map
$\Phi(z)=|z|^{p-2}z$ is inverted away from $z=0$, its only singular point.

Since $v'<0$ we have $|v'|^{p-2}v'=-(-v')^{p-1}$ and $-v'=\frac{x v}{r}$, so
\[
r^{N-1}(-v')^{p-1}=r^{N-p}x^{p-1}v^{p-1}.
\]
Equation \eqref{eq:ode} reads
$\big(r^{N-1}(-v')^{p-1}\big)'=r^{N-1-p}v^{p-1}\big[\mu+r^{p+s}v^{q-p+1}\big]
=r^{N-1-p}v^{p-1}(\mu+Y)$, that is
\[
\big(r^{N-p}x^{p-1}v^{p-1}\big)'=r^{N-p-1}v^{p-1}(\mu+Y).
\]
Expanding the left-hand side and using $\frac{rv'}{v}=-x$,
$\frac{r x'}{x}=\frac{\dot x}{x}$, we get
\[
x^{p-1}\Big[(N-p)+(p-1)\frac{\dot x}{x}-(p-1)x\Big]=\mu+Y ,
\]
i.e.\ $(p-1)x^{p-2}\dot x=\mu+Y-(N-p)x^{p-1}+(p-1)x^{p}=\mu+Y-h(x)$, which is
the first equation of \eqref{eq:system}. For the second, differentiating
$Y=r^{p+s}v^{q-p+1}$ with respect to $t$ gives
\[
\dot Y=(p+s)Y+(q-p+1)r^{p+s}v^{q-p}\,r v'
=Y\big[(p+s)-(q-p+1)x\big]=(q-p+1)Y(\gs-x).
\]
Conversely, let $(x,Y)$ solve \eqref{eq:system} on $J$ with $x>0$, $Y>0$, and
define $v>0$ on $\{\mathrm e^{t}:t\in J\}$ by the first identity in
\eqref{eq:recover}. Then $\ln v=\frac{1}{q-p+1}\ln Y-\gs t$, so the second
equation of \eqref{eq:system} gives
\[
\frac{r v'(r)}{v(r)}=\frac{d}{dt}\ln v
=\frac{1}{q-p+1}\frac{\dot Y}{Y}-\gs=(\gs-x)-\gs=-x<0 ,
\]
which is the second identity in \eqref{eq:recover} and shows that $v'<0$;
reading the computation above backwards, the first equation of
\eqref{eq:system} is then equivalent to \eqref{eq:ode}.
\end{proof}

\begin{remark}\label{rem:analytic}
The vector field in \eqref{eq:system} is real-analytic on
$\{(x,Y):x>0\}$ for every $p>1$: the only possible singularity is the factor
$x^{2-p}$, which is analytic for $x>0$. This is in sharp contrast with the
formulations in the variables $(w,w')$ or $(x_{l},y_{l})$ of
\eqref{eq:francasystem}, in which one has to invert $\Phi(z)=|z|^{p-2}z$: the
inverse $\Phi^{-1}(z)=|z|^{\frac{2-p}{p-1}}z$ is only H\"older continuous at
$z=0$ when $p>2$, and it is exactly this that forces the standing restriction
$p\le2$ in \cite[\S2]{Franca2008} and \cite[Rem.~7]{DFS2026}. All the classical
local theory, stable manifold, centre manifold, Hartman--Grobman, is
therefore available here without any restriction on $p$.
\end{remark}

We shall also use the following elementary uniqueness property, which lets us
propagate backwards an identity valid only for large $r$.

\begin{remark}\label{rem:uniqueness}
Set $\Phi(z)=|z|^{p-2}z$ and $W=r^{N-1}\Phi(v')$. Then \eqref{eq:ode} is
equivalent to the first order system
\[
v'=\Phi^{-1}\big(r^{1-N}W\big),
\qquad
W'=-\big(\mu\,r^{N-1-p}v^{p-1}+r^{N-1+s}v^{q}\big),
\]
whose right-hand side is of class $C^{1}$, hence locally Lipschitz, on
$\{r>0,\ v>0,\ W\ne0\}$, because $\Phi^{-1}(z)=|z|^{\frac{2-p}{p-1}}z$ is
smooth away from $z=0$. Consequently the initial value problem is uniquely
solvable in both time directions as long as $v>0$ and $v'\ne0$, so that two
such solutions which agree on a subinterval agree on their whole common
interval of definition. In particular, a solution of \eqref{eq:ode} on
$(R_{0},+\infty)$ which coincides with $A_{*}r^{-\gs}$ for all large $r$
coincides with it on the whole of $(R_{0},+\infty)$.
\end{remark}

We record the linearisations at the three equilibria \eqref{eq:equilibria}.
Write $c:=q-p+1>0$ and recall
$h'(\gamma)=(p-1)\gamma^{p-2}\big(N-p-p\gamma\big)$.

\begin{lemma}\label{lem:linear}
Assume \eqref{eq:standing}.
\begin{enumerate}
\item[(i)] At $P_{\pm}=(\nu_{\pm},0)$ the Jacobian of \eqref{eq:system} is upper
triangular with eigenvalues
\[
\lambda_{1}^{\pm}=-\frac{\nu_{\pm}^{2-p}\,h'(\nu_{\pm})}{p-1}
=-\big(N-p-p\nu_{\pm}\big),
\qquad
\lambda_{2}^{\pm}=c\,(\gs-\nu_{\pm}) ,
\]
and $\lambda_{1}^{-}<0<\lambda_{1}^{+}$. The eigenvector associated with
$\lambda_{1}^{\pm}$ is $(1,0)$, so the invariant line $\{Y=0\}$ carries the
one-dimensional invariant manifold associated with $\lambda_{1}^{\pm}$, which
is stable at $P_{-}$ and unstable at $P_{+}$. Moreover:
\begin{itemize}
\item $\lambda_{2}^{-}>0$ if $\gs>\nu_{-}$ (then $P_{-}$ is a saddle whose local
stable manifold is contained in $\{Y=0\}$), and $\lambda_{2}^{-}<0$ if
$\gs<\nu_{-}$ (then $P_{-}$ is a stable node);
\item $\lambda_{2}^{+}<0$ if $\gs<\nu_{+}$ (regimes {\rm(A)--(C)}; then $P_{+}$
is a saddle whose local stable manifold is one-dimensional and transversal to
$\{Y=0\}$), $\lambda_{2}^{+}>0$ if $\gs>\nu_{+}$ (then $P_{+}$ is an unstable
node), and $\lambda_{2}^{+}=0$ if $\gs=\nu_{+}$.
\end{itemize}
\item[(ii)] If $\nu_{-}<\gs<\nu_{+}$, the Jacobian at $P_{*}$ has trace
$p\gs-(N-p)$ and determinant $\frac{c\,Y_{*}\gs^{2-p}}{p-1}>0$. Hence $P_{*}$ is
hyperbolic and asymptotically stable in regime {\rm(A)}, hyperbolic and
repelling in regime {\rm(C)}, and a linear centre in regime {\rm(B)}.
\end{enumerate}
\end{lemma}

\begin{proof}
Write $F(x,Y)=\frac{x^{2-p}}{p-1}(\mu+Y-h(x))$ and
$G(x,Y)=cY(\gs-x)$. Then
$F_{Y}=\frac{x^{2-p}}{p-1}$, $G_{x}=-cY$, $G_{Y}=c(\gs-x)$, and
\[
F_{x}=\frac{(2-p)x^{1-p}}{p-1}\big(\mu+Y-h(x)\big)-\frac{x^{2-p}h'(x)}{p-1}.
\]
At $P_{\pm}$ we have $Y=0$ and $h(\nu_{\pm})=\mu$, so the first term vanishes
and $F_{x}=-\frac{\nu_{\pm}^{2-p}h'(\nu_{\pm})}{p-1}
=-(N-p-p\nu_{\pm})$ by \eqref{eq:hgamma}; also $G_{x}=0$, which gives (i)
after recalling \eqref{eq:nupm} and \eqref{eq:qsuper}. At $P_{*}$ we have
$\mu+Y_{*}-h(\gs)=0$, so $F_{x}=-\frac{\gs^{2-p}h'(\gs)}{p-1}=-(N-p-p\gs)$ and
$G_{Y}=0$, whence $\mathrm{tr}=F_{x}+G_{Y}=p\gs-(N-p)$; its sign is that of
$\gs-\frac{N-p}{p}$, hence negative in regime {\rm(A)}, zero in regime {\rm(B)}
and positive in regime {\rm(C)}, by \eqref{eq:qsuper}. Finally
$\det=F_{x}G_{Y}-F_{Y}G_{x}=0+\frac{\gs^{2-p}}{p-1}cY_{*}>0$.
\end{proof}

%%%%%%%%%%%%%%%%%%%%%%%%%%%%%%%%%%%%%%%%%%%%%%%%%%%%%%%%%%%%%%%%%%%%%%%%%%%%%%%
\section{A priori bounds: proof of Theorem~\ref{thm:apriori}}\label{sec:apriori}
%%%%%%%%%%%%%%%%%%%%%%%%%%%%%%%%%%%%%%%%%%%%%%%%%%%%%%%%%%%%%%%%%%%%%%%%%%%%%%%

We begin with \eqref{eq:auto}, which is where the source terms on the
right-hand side of \eqref{eq:ode} enter.

\begin{lemma}\label{lem:auto}
Assume \eqref{eq:standing} and let $v>0$ solve \eqref{eq:ode} on
$(R_{0},+\infty)$. Then $v'(r)<0$ for all large $r$ and
$\lim_{r\to+\infty}v(r)=0$.
\end{lemma}

\begin{proof}
Set $W(r):=r^{N-1}|v'|^{p-2}v'$, so that $\sgn W=\sgn v'$ and, by
\eqref{eq:ode},
\begin{equation}\label{eq:Wprime}
W'(r)=-\big(\mu\,r^{N-1-p}v^{p-1}+r^{N-1+s}v^{q}\big)<0 ,
\end{equation}
so $W$ is strictly decreasing.

\emph{Step 1: $v'<0$ for large $r$.} If $W(r_{1})\le0$ for some
$r_{1}>R_{0}$, then $W<0$, that is $v'<0$, on $(r_{1},+\infty)$ and we are
done. Otherwise $W>0$ on $(R_{0},+\infty)$, so $v'>0$ there; fixing
$R_{1}>R_{0}$ we get $v\ge m:=v(R_{1})>0$ on $[R_{1},+\infty)$, whence, by
\eqref{eq:Wprime} and $N>p$,
\[
W(r)\ \le\ W(R_{1})-\mu m^{p-1}\!\!\int_{R_{1}}^{r}\!\rho^{N-1-p}\,d\rho
=W(R_{1})-\frac{\mu m^{p-1}}{N-p}\big(r^{N-p}-R_{1}^{N-p}\big)
\ \longrightarrow\ -\infty ,
\]
contradicting $W>0$.

\emph{Step 2: $v(r)\to0$.} By Step 1 there is $r_{1}>R_{0}$ with $v'<0$ on
$(r_{1},+\infty)$, so $L:=\lim_{r\to+\infty}v(r)\ge0$ exists. Suppose $L>0$.
Then $v\ge L$ on $(r_{1},+\infty)$ and, as above,
$-W(r)\ge\frac{\mu L^{p-1}}{N-p}\big(r^{N-p}-r_{1}^{N-p}\big)$, so there are
$C>0$ and $r_{2}\ge r_{1}$ with $r^{N-1}(-v')^{p-1}\ge C\,r^{N-p}$ for
$r\ge r_{2}$. Hence $(-v')^{p-1}\ge C r^{1-p}$, i.e.
$-v'\ge C^{\frac{1}{p-1}}r^{-1}$, and integrating,
$v(r_{2})-v(r)\ge C^{\frac{1}{p-1}}\ln(r/r_{2})\to+\infty$, contradicting
$v>0$. Therefore $L=0$.
\end{proof}

\begin{proof}[Proof of Theorem~\ref{thm:apriori}]
Assertion \eqref{eq:auto} is Lemma~\ref{lem:auto}. By that lemma there is
$R_{1}>R_{0}$ with $v'<0$ on $(R_{1},+\infty)$, so Lemma~\ref{lem:system}
applies: $(x,Y)$ solves \eqref{eq:system} on $[t_{0},+\infty)$ with $x>0$,
$Y>0$, where $t_{0}=\ln R_{1}$. It remains to prove \eqref{eq:apriori}.

\smallskip
\noindent\emph{Step 1: $x$ stays away from $0$.} Fix $\delta\in(0,\nu_{-})$;
then $h$ is increasing on $(0,\delta]$ and $h(\delta)<h(\nu_{-})=\mu$. Since
$x^{2-p}>0$ for $x>0$, on $\{0<x\le\delta\}$ we have
\[
\dot x=\frac{x^{2-p}}{p-1}\big[\mu+Y-h(x)\big]
\ \ge\ \frac{x^{2-p}}{p-1}\big(\mu-h(\delta)\big)\ >\ 0 .
\]
Consequently, for every $\eta\in(0,\delta]$ the set $\{x>\eta\}$ is forward
invariant: an orbit reaching $x=\eta$ must cross it upwards. Taking
$\eta:=\min\{x(t_{0}),\delta\}>0$ we obtain $x(t)\ge\eta$ for all $t\ge t_{0}$,
whence $\liminf_{t\to\infty}x(t)\ge\eta>0$.

\smallskip
\noindent\emph{Step 2: $x(t)<\nu_{+}$ for all $t\ge t_{0}$.} We first note that
$x$ cannot blow up in finite time. Indeed $v>0$ and $v\in C^{1}$ on
$(R_{0},+\infty)$, so $x=\frac{r(-v')}{v}$ is continuous on
$(R_{0},+\infty)$ and therefore bounded on every compact subinterval; in
particular $x$ is finite at every finite $t$.

Suppose now, by contradiction, that $x(t_{1})\ge\nu_{+}$ for some
$t_{1}\ge t_{0}$. Since $h$ is strictly decreasing on
$\big(\frac{N-p}{p},\infty\big)$ and $\nu_{+}>\frac{N-p}{p}$, we have
$h(x)\le\mu$ for $x\ge\nu_{+}$, hence
\[
\dot x=\frac{x^{2-p}}{p-1}\big[\mu+Y-h(x)\big]
\ \ge\ \frac{x^{2-p}}{p-1}\,Y\ >\ 0 .
\]
Therefore $x$ is strictly increasing on $[t_{1},\infty)$ and stays above
$\nu_{+}$. Fix $t_{2}>t_{1}$ and put $x_{2}:=x(t_{2})>\nu_{+}$. For
$t\ge t_{2}$,
\[
\dot x\ \ge\ \frac{x^{2-p}}{p-1}\big[\mu-h(x)\big]
=\frac{x^{2-p}}{p-1}\Big[\mu-(N-p)x^{p-1}+(p-1)x^{p}\Big]=:\Theta(x) ,
\]
where $\Theta>0$ on $(\nu_{+},\infty)$ and $\Theta(x)=x^{2}\big(1+o(1)\big)$ as
$x\to+\infty$; in particular
$\int_{x_{2}}^{\infty}\frac{dx}{\Theta(x)}<+\infty$. Comparison with the
scalar equation $\dot z=\Theta(z)$, $z(t_{2})=x_{2}$, then shows that $x$
blows up at a finite time
$t^{*}\le t_{2}+\int_{x_{2}}^{\infty}\frac{dx}{\Theta(x)}$, contradicting the
first paragraph. Hence $x(t)<\nu_{+}$ for all $t\ge t_{0}$.

\smallskip
\noindent\emph{Step 3: $Y$ is bounded.} By Steps 1 and 2 the orbit satisfies
$x(t)\in[\eta,\nu_{+}]$ for all $t\ge t_{0}$, a compact subinterval of
$(0,+\infty)$; hence $\frac{x(t)^{2-p}}{p-1}\ge\vartheta$ for some
$\vartheta>0$. As $h\le\mub$,
on the set $\mathcal Y:=\{t\ge t_{0}:\ Y(t)\ge\mub\}$ we have
$\mu+Y-h(x)\ge\mu>0$, hence
\[
\dot x\ \ge\ \vartheta\,\mu\ =:c_{1}>0\qquad\text{on }\mathcal Y .
\]
Let $J$ be any connected component of $\mathcal Y$ and let $a\ge t_{0}$ be its
left endpoint. Since $x$ increases on $J$ and takes values in
$[\eta,\nu_{+}]$, for every $\tau\in J$ we get
$c_{1}(\tau-a)\le x(\tau)-x(a)\le\nu_{+}-\eta$; hence $J$ is bounded and
\[
|J|\ \le\ T_{1}:=\frac{\nu_{+}-\eta}{c_{1}} .
\]
On the other hand $\dot Y=cY(\gs-x)\le c\,\gs\,Y$ everywhere, so
$Y(t)\le Y(a)e^{c\gs T_{1}}$ for $t\in J$. If $a>t_{0}$ then $Y(a)=\mub$ by
continuity, whence $\sup_{J}Y\le\mub\,e^{c\gs T_{1}}$; if $a=t_{0}$ then
$\sup_{J}Y\le Y(t_{0})e^{c\gs T_{1}}$. Since
$Y<\mub$ outside $\mathcal Y$, we conclude
\[
\sup_{t\ge t_{0}}Y(t)\ \le\
\max\big\{\mub,\ Y(t_{0})\big\}\,e^{c\gs T_{1}}<+\infty .
\]

\end{proof}

\begin{remark}\label{rem:barrier}
Step~2 above also justifies calling $x=\nu_{+}$ a \emph{barrier}. Let $v>0$
solve \eqref{eq:ode} with $v'(r_{1})<0$ and
$x(t_{1})=-\frac{r_{1}v'(r_{1})}{v(r_{1})}\ge\nu_{+}$ at some
$r_{1}=\mathrm e^{t_{1}}$, and let $(r_{1},r^{*})$ be the maximal interval to
the right of $r_{1}$ on which $v>0$ and $v'<0$. Then $r^{*}<+\infty$ and
$v(r^{*})=0$. Indeed, if $r^{*}=+\infty$ then Lemma~\ref{lem:system} applies on
$(t_{1},+\infty)$ and Step~2, run from $t_{1}$ on, makes $x$ blow up at a finite
time, whereas $x$ is finite wherever $v>0$; hence $r^{*}<+\infty$. On
$[r_{1},r^{*})$ the function $v$ decreases with values in $(0,v(r_{1})]$, while
$W=r^{N-1}|v'|^{p-2}v'$ decreases and obeys
$|W'|\le\mu\,r^{N-1-p}v(r_{1})^{p-1}+r^{N-1+s}v(r_{1})^{q}$, which is bounded
there; so $(v,W)$ extends continuously to $r^{*}$, with
$W(r^{*})\le W(r_{1})<0$ and hence $v'(r^{*})<0$. Maximality then forces
$v(r^{*})=0$. Thus an orbit crossing $x=\nu_{+}$ corresponds to a solution
vanishing at the finite radius $r^{*}$, as asserted after
Theorem~\ref{thm:apriori} and in Figure~\ref{fig:portraits}{\rm(D)}.
\end{remark}

%%%%%%%%%%%%%%%%%%%%%%%%%%%%%%%%%%%%%%%%%%%%%%%%%%%%%%%%%%%%%%%%%%%%%%%%%%%%%%%

%%%%%%%%%%%%%%%%%%%%%%%%%%%%%%%%%%%%%%%%%%%%%%%%%%%%%%%%%%%%%%%%%%%%%%%%%%%%%%%
\section{Absence of cycles: proof of Theorem~\ref{thm:dulac}}\label{sec:dulac}
%%%%%%%%%%%%%%%%%%%%%%%%%%%%%%%%%%%%%%%%%%%%%%%%%%%%%%%%%%%%%%%%%%%%%%%%%%%%%%%

Throughout this section $\mathbf F=(F,G)$ is the vector field \eqref{eq:FG} of
\eqref{eq:system}, $c=q-p+1>0$, and $B$, $b$ are as in \eqref{eq:B}. Note that
$b>-1$: indeed $b+1=\frac{N-p}{p+s}>0$ because $N>p$ and $s>-p$.

\begin{proof}[Proof of \eqref{eq:divB} and of Theorem~\ref{thm:dulac}(ii)]
Since $\gs\,c=p+s$, we have $b+1=\frac{N-p}{\gs c}$, that is
\begin{equation}\label{eq:beta}
c\,(b+1)=\beta:=\frac{N-p}{\gs}.
\end{equation}
Now
\[
B\,F=x^{p-2}Y^{b}\cdot\frac{x^{2-p}}{p-1}\big[\mu+Y-h(x)\big]
=\frac{Y^{b}}{p-1}\big[\mu+Y-h(x)\big],
\qquad
B\,G=c\,x^{p-2}\,Y^{b+1}(\gs-x),
\]
so that
\[
\partial_{x}(BF)=-\frac{Y^{b}h'(x)}{p-1},
\qquad
\partial_{Y}(BG)=c(b+1)\,x^{p-2}Y^{b}(\gs-x)
=\beta\,x^{p-2}Y^{b}(\gs-x).
\]
Using $h'(x)=(p-1)x^{p-2}\big(N-p-px\big)$ we obtain
\[
\dv(B\mathbf F)=Y^{b}x^{p-2}\Big[-\big(N-p-px\big)+\beta(\gs-x)\Big]
=Y^{b}x^{p-2}\Big[(p-\beta)x+\beta\gs-(N-p)\Big].
\]
By \eqref{eq:beta}, $\beta\gs=N-p$, so the bracket equals $(p-\beta)x$ and
\[
\dv(B\mathbf F)=\Big(p-\frac{N-p}{\gs}\Big)x^{p-1}Y^{b},
\]
which is \eqref{eq:divB}. Finally the sign of $p-\frac{N-p}{\gs}$ is that of
$\gs-\frac{N-p}{p}$, so it is negative, zero or positive according to whether
$q>p^{*}_{s}-1$, $q=p^{*}_{s}-1$ or $q<p^{*}_{s}-1$, by \eqref{eq:qsuper}.

Assume now $q=p^{*}_{s}-1$, i.e.\ $\gs=\frac{N-p}{p}$, so that $\beta=p$ and
$b+1=\frac{p}{c}$ by \eqref{eq:beta}. With $\mathcal H$ as in
\eqref{eq:firstintegral} we have
\[
\partial_{Y}\mathcal H=\frac{Y^{b}}{p-1}\big[\mu-h(x)+Y\big]=B\,F ,
\]
and, since $h'(x)=(p-1)x^{p-2}(N-p-px)=p(p-1)x^{p-2}(\gs-x)$ and
$\frac{1}{b+1}=\frac cp$,
\[
\partial_{x}\mathcal H=-\frac{h'(x)\,Y^{b+1}}{(p-1)(b+1)}
=-\frac{p(\gs-x)x^{p-2}Y^{b+1}}{b+1}=-c\,x^{p-2}Y^{b+1}(\gs-x)=-B\,G .
\]
Hence $\frac{d}{dt}\mathcal H(x,Y)=\partial_{x}\mathcal H\,F
+\partial_{Y}\mathcal H\,G=B^{-1}\big(-BG\cdot BF+BF\cdot BG\big)=0$, so
$\mathcal H$ is a first integral. This proves part (ii) of
Theorem~\ref{thm:dulac}.
\end{proof}

\begin{remark}
The choice \eqref{eq:B} is not accidental. Looking for a Dulac function of the
monomial form $B=x^{a}Y^{b}$, one finds
\[
\dv(B\mathbf F)=\frac{Y^{b}x^{a+1-p}}{p-1}
\Big[(a+2-p)\big(\mu+Y-h(x)\big)-x\,h'(x)\Big]
+c(b+1)x^{a}Y^{b}(\gs-x).
\]
The bracket is the only term containing $Y$; in order for the $Y$-dependence of
$\dv(B\mathbf F)$ to reduce to the harmless positive factor $Y^{b}$, so that
the sign of the divergence is a function of $x$ alone, one is forced to take
$a=p-2$. The divergence is then $Y^{b}x^{p-2}$ times the affine function
$x\mapsto\big(p-c(b+1)\big)x+c(b+1)\gs-(N-p)$, and the choice \eqref{eq:beta} of
$b$ is the unique one for which this affine function vanishes at $x=0$, hence
keeps a constant sign on the whole of $(0,+\infty)$.
\end{remark}

\begin{proposition}\label{prop:nocycle}
Assume \eqref{eq:standing} and $q\ne p^{*}_{s}-1$. Then \eqref{eq:system} has
no periodic orbit
contained in $\{x>0,\ Y\ge0\}$, and no separatrix cycle (that is, no compact
invariant set consisting of finitely many equilibria together with orbits
connecting them and bounding a region) contained in $\{x>0,\ Y\ge0\}$.
\end{proposition}

\begin{proof}
\emph{Periodic orbits.} The line $\{Y=0\}$ is invariant, and on it the flow is
the scalar equation $\dot x=\frac{x^{2-p}}{p-1}(\mu-h(x))$, which has no
periodic orbit. Hence a periodic orbit $\Gamma$ would meet $\{Y>0\}$ and,
$\{Y=0\}$ being invariant, would be entirely contained in the open simply
connected quadrant $Q:=\{x>0,\ Y>0\}$, on which $B\in C^{1}$, $B>0$. Let $D$ be
the bounded open region enclosed by $\Gamma$; then $\overline D\subset Q$. Since
$\mathbf F$ is tangent to $\Gamma$, the divergence theorem gives
\[
0=\oint_{\Gamma}\big(B\mathbf F\big)\cdot\mathbf n\,d\sigma
=\iint_{D}\dv\big(B\mathbf F\big)\,dx\,dY\ \ne\ 0
\]
by \eqref{eq:divB}, because $\dv(B\mathbf F)$ has a strict sign on $Q$ and $D$
has positive measure; a contradiction.

\smallskip
\emph{Separatrix cycles.} Let $\Gamma$ be such a cycle, bounding a region $D$.
If $\Gamma\subset Q$ the previous argument applies verbatim. Otherwise
$\Gamma$ meets $\{Y=0\}$; since $\{Y=0\}$ is invariant and $\Gamma$ is a union
of orbits and equilibria, $\Gamma\cap\{Y=0\}$ is a union of equilibria and of
subintervals of $\{Y=0\}$. Moreover $D\subset Q$: indeed $\Gamma$ is a compact
subset of $\{x>0,\ Y\ge0\}$, so $\overline D\subset\{x>0,\ Y\ge0\}$, and an
interior point of $D$ with $Y=0$ would have a whole neighbourhood inside
$\{Y\ge0\}$, which is absurd. For $\eps>0$ put
$D_{\eps}:=D\cap\{Y>\eps\}$. For a.e.\ small $\eps>0$ the set $D_{\eps}$ is a
bounded open set whose boundary consists of arcs of $\Gamma$, along which
$\mathbf F$ is tangent, together with finitely many horizontal segments
contained in $\{Y=\eps\}$, along which the outward normal is $\mathbf n=(0,-1)$
and hence
\[
\big(B\mathbf F\big)\cdot\mathbf n=-B\,G
=-c\,x^{p-2}\,\eps^{\,b+1}\,(\gs-x).
\]
Since $x$ is bounded on $\overline D$ and $b+1>0$, the total flux through these
segments is $O(\eps^{\,b+1})\to0$ as $\eps\to0^{+}$. The divergence theorem on
$D_{\eps}$ therefore gives
\[
\iint_{D_{\eps}}\dv\big(B\mathbf F\big)\,dx\,dY=O\big(\eps^{\,b+1}\big).
\]
On the other hand, $\dv(B\mathbf F)=\big(p-\frac{N-p}{\gs}\big)x^{p-1}Y^{b}$ has
a strict sign and, since $b>-1$ and $\overline D$ is a bounded subset of
$\{x>0\}$, it is absolutely integrable on $D$; by monotone convergence
$\iint_{D_{\eps}}\dv(B\mathbf F)\to\iint_{D}\dv(B\mathbf F)\ne0$. This
contradiction proves the claim.
\end{proof}

\begin{proof}[Proof of Theorem~\ref{thm:dulac}]
Identity \eqref{eq:divB} and part (ii) were proved above. Assume
$q\ne p^{*}_{s}-1$; Proposition~\ref{prop:nocycle} excludes periodic orbits and
separatrix cycles. Let now $(x,Y)$ be a solution with $x>0$, $Y>0$ satisfying
\eqref{eq:apriori}: there are $\delta>0$, $M>0$ and $t_{0}$ with
\[
(x(t),Y(t))\in K:=[\delta,\nu_{+}]\times[0,M]\qquad\text{for }t\ge t_{0},
\]
and $K$ is a compact subset of the half-plane $(0,+\infty)\times\R$ on which
$\mathbf F$ is analytic. Hence the $\omega$-limit set $\omega$ of the orbit is a
nonempty, compact, connected, invariant subset of $K$. By the
Poincar\'e--Bendixson theorem \cite[\S1.7]{DLA}, $\omega$ is either a single
equilibrium, or a
periodic orbit, or consists of finitely many equilibria together with orbits
joining them. In the last case $\omega$ is a separatrix cycle in the sense of
Proposition~\ref{prop:nocycle}: if $\Gamma_{1}\subset\omega$ is a nonstationary
orbit, its $\alpha$- and $\omega$-limit sets are equilibria of $\omega$, and
since the orbit $(x,Y)$ accumulates on every point of $\omega$ it must return
to a neighbourhood of $\alpha(\Gamma_{1})$ after following $\Gamma_{1}$; the
usual argument with a transversal at a point of $\Gamma_{1}$, together with the
monotonicity of the successive intersections with it, then shows that the
equilibria and the connecting orbits of $\omega$ close up into a compact curve
bounding a region (see also \cite[\S1.7]{DLA}). The last two alternatives are
excluded by Proposition~\ref{prop:nocycle}, so $\omega=\{P\}$ for an equilibrium $P$ of
\eqref{eq:system} lying in $K$; by \eqref{eq:equilibria} and $\delta>0$,
$P\in\{P_{*},P_{-},P_{+}\}$.
\end{proof}

%%%%%%%%%%%%%%%%%%%%%%%%%%%%%%%%%%%%%%%%%%%%%%%%%%%%%%%%%%%%%%%%%%%%%%%%%%%%%%%

%%%%%%%%%%%%%%%%%%%%%%%%%%%%%%%%%%%%%%%%%%%%%%%%%%%%%%%%%%%%%%%%%%%%%%%%%%%%%%%
\section{The Lyapunov function: proof of Theorem~\ref{thm:lyap}}\label{sec:lyap}
%%%%%%%%%%%%%%%%%%%%%%%%%%%%%%%%%%%%%%%%%%%%%%%%%%%%%%%%%%%%%%%%%%%%%%%%%%%%%%%

\begin{proof}[Proof of Theorem~\ref{thm:lyap}]
Write $c=q-p+1$ and
$\Xi(x):=c(p-1)\big[\frac{x^{p}}{p}-\gs\frac{x^{p-1}}{p-1}\big]$, so that
$\Xi'(x)=c(p-1)(x-\gs)x^{p-2}$ and $V=\Xi(x)+\big(Y-Y_{*}\ln Y\big)$.
Along \eqref{eq:system},
\[
\frac{d}{dt}\Xi(x)=\Xi'(x)\dot x
=c(p-1)(x-\gs)x^{p-2}\cdot\frac{x^{2-p}}{p-1}\big[\mu+Y-h(x)\big]
=c\,(x-\gs)\big[\mu+Y-h(x)\big],
\]
\[
\frac{d}{dt}\big(Y-Y_{*}\ln Y\big)=\Big(1-\frac{Y_{*}}{Y}\Big)\dot Y
=c\,(Y-Y_{*})(\gs-x).
\]
Adding and using $\mu+Y_{*}=h(\gs)$,
\[
\dot V=c(x-\gs)\big[\mu+Y-h(x)-(Y-Y_{*})\big]
=c(x-\gs)\big[h(\gs)-h(x)\big],
\]
which is \eqref{eq:Vdot}.

Since $\gs<\frac{N-p}{p}$ and $h$ is strictly increasing on
$\big(0,\frac{N-p}{p}\big)$ and strictly decreasing afterwards, the number
$\gs^{\sharp}>\frac{N-p}{p}$ with $h(\gs^{\sharp})=h(\gs)$ is well defined and
\[
\big(x-\gs\big)\big(h(x)-h(\gs)\big)>0\qquad
\text{for } x\in(0,\gs^{\sharp})\setminus\{\gs\} .
\]
Hence $\dot V<0$ on $\{0<x<\gs^{\sharp}\}\setminus\{x=\gs\}$, while $\dot V=0$
on that strip exactly on the segment $\{x=\gs\}$. We stress that $\dot V$
vanishes on the whole of that segment, and not only at $P_{*}$; the conclusion
will therefore be reached through LaSalle's invariance principle rather than
through strict decrease.

We now describe the sublevel sets of $V$. The function
$\zeta(Y):=Y-Y_{*}\ln Y$ is strictly convex on $(0,\infty)$, has its strict
minimum at $Y=Y_{*}$, and $\zeta(Y)\to+\infty$ both as $Y\to0^{+}$ and as
$Y\to+\infty$. As for $\Xi$, we have $\Xi'(x)=c(p-1)(x-\gs)x^{p-2}$, so
$\Xi$ is strictly decreasing on $(0,\gs)$, strictly increasing on
$(\gs,\infty)$, and $\Xi(x)\to+\infty$ as $x\to+\infty$; in particular $\Xi$
has a strict global minimum at $x=\gs$ and is bounded below on $(0,\infty)$.
(We stress that $\Xi$ need \emph{not} be convex: for $p>2$ one computes
$\Xi''(x)=c(p-1)x^{p-3}\big[(p-1)x-(p-2)\gs\big]$, which is negative for
$x<\frac{(p-2)\gs}{p-1}$. Only the monotonicity just stated is used.)

Consequently $V=\Xi+\zeta$ has a strict global minimum at $P_{*}$, and, since
$\zeta$ attains its minimum at $Y_{*}$,
\[
c_{0}=V(\gs^{\sharp},Y_{*})=\min\big\{V(x,Y):\ x=\gs^{\sharp},\ Y>0\big\}.
\]
Let $\mathcal W$ be the connected component of $\{V<c_{0}\}$ containing
$P_{*}$. By the displayed identity, $\mathcal W$ does not meet
$\{x=\gs^{\sharp}\}$, hence $\mathcal W\subset\{0<x<\gs^{\sharp}\}$, so that
$\dot V<0$ on $\mathcal W\setminus\{x=\gs\}$ and $\mathcal W$ is forward
invariant. Moreover $\mathcal W$ is bounded: its $x$-range lies in
$(0,\gs^{\sharp})$, and on $\mathcal W$ one has
$\zeta(Y)<c_{0}-\inf_{(0,\infty)}\Xi<+\infty$, which confines $Y$ to a compact
subinterval $[Y_{1},Y_{2}]\subset(0,\infty)$. (The closure
$\overline{\mathcal W}$ may well meet the line $\{x=0\}$, namely when
$\Xi(\gs^{\sharp})>0$; this is harmless, since only the compact set
$\overline{\mathcal W}\cap\{x\ge\delta\}$ will be used below, with $\delta$
supplied by Theorem~\ref{thm:apriori}.)

Let now an orbit enter $\mathcal W$, say at time $\tau$. It remains there for
all later times, and by Theorem~\ref{thm:apriori} it also satisfies
$x(t)\ge\delta$ for some $\delta>0$; hence it remains in the compact set
$\overline{\mathcal W}\cap\{x\ge\delta\}\subset\{x>0,\ Y>0\}$, and its
$\omega$-limit set $\omega$ is nonempty, compact, connected and invariant.
Since $V$ is nonincreasing along the orbit,
$V\le V\big(x(\tau),Y(\tau)\big)<c_{0}$ on $\omega$,
so $\omega\subset\mathcal W$; and $V$ is constant on $\omega$, so $\dot V=0$
there, whence $\omega\subset\{x=\gs\}\cap\mathcal W$. On that segment
$\dot x=\frac{\gs^{2-p}}{p-1}(Y-Y_{*})$ vanishes only at $Y=Y_{*}$, so the only
invariant subset is $\{P_{*}\}$. Hence $\omega=\{P_{*}\}$, i.e.\
$x(t)\to\gs$ and $Y(t)\to Y_{*}$. By \eqref{eq:recover},
$r^{\gs}v(r)=Y^{1/c}\to Y_{*}^{1/c}=A_{*}$ and $\frac{rv'}{v}=-x\to-\gs$.
\end{proof}

\begin{remark}\label{rem:cycles}
The Lyapunov function $V$ is \emph{not} decreasing on the whole quadrant: the
set where $\dot V>0$ is $\{x>\gs^{\sharp}\}$, which is nonempty inside the
region $\{x<\nu_{+}\}$ where orbits live, since $h(\gs^{\sharp})=h(\gs)>\mu
=h(\nu_{+})$ forces $\gs^{\sharp}<\nu_{+}$. Integrating \eqref{eq:Vdot} over a
hypothetical periodic orbit shows that such an orbit would have to visit
$\{x>\gs^{\sharp}\}$, so $V$ alone does not exclude cycles. This is precisely
what the Dulac function of Theorem~\ref{thm:dulac} does, and it is why both
tools are needed: $B$ rules out cycles globally, while $V$ provides the
quantitative basin \eqref{eq:basin}.
\end{remark}

%%%%%%%%%%%%%%%%%%%%%%%%%%%%%%%%%%%%%%%%%%%%%%%%%%%%%%%%%%%%%%%%%%%%%%%%%%%%%%%
\section{Exact decay rates: proof of Theorem~\ref{thm:main}}\label{sec:main}
%%%%%%%%%%%%%%%%%%%%%%%%%%%%%%%%%%%%%%%%%%%%%%%%%%%%%%%%%%%%%%%%%%%%%%%%%%%%%%%

Let $v$ be as in Theorem~\ref{thm:main}. By Theorem~\ref{thm:apriori} we have
$v'<0$ for large $r$, the bounds \eqref{eq:apriori} hold, and by
Theorem~\ref{thm:dulac} the orbit $(x,Y)$ converges to one of $P_{*}$,
$P_{-}$, $P_{+}$. We first record which limits are possible, then compute the
corresponding decay rates.

\begin{lemma}\label{lem:which}
Let $(x,Y)$ be the orbit of a solution as above, so $Y(t)>0$ for all $t$.
\begin{enumerate}
\item[(i)] If $\gs>\nu_{-}$, the limit $P_{-}$ is impossible.
\item[(ii)] If $\gs<\nu_{-}$, the equilibrium $P_{*}$ does not exist in
$\{Y>0\}$, so the limit is $P_{-}$ or $P_{+}$.
\item[(iii)] If $\gs=\nu_{-}$, then $P_{*}=P_{-}=(\gs,0)$, and the limit is
$(\gs,0)$ or $P_{+}$.
\end{enumerate}
\end{lemma}

\begin{proof}
(i) If $\gs>\nu_{-}$ then, by Lemma~\ref{lem:linear}(i),
$\lambda_{1}^{-}=-(N-p-p\nu_{-})<0$ and $\lambda_{2}^{-}=c(\gs-\nu_{-})>0$, so
$P_{-}$ is a hyperbolic saddle. Its eigenvector for $\lambda_{1}^{-}$ is
$(1,0)$, so the local stable manifold of $P_{-}$ is a curve tangent to the
invariant line $\{Y=0\}$ at $P_{-}$; since $\{Y=0\}$ is invariant and
one-dimensional, the local stable manifold \emph{is} a neighbourhood of
$P_{-}$ in $\{Y=0\}$. An orbit with $Y>0$ therefore cannot converge to
$P_{-}$.

(ii) If $\gs<\nu_{-}$ then $Y_{*}=h(\gs)-\mu<0$ by
\eqref{eq:trichotomycases}, so $P_{*}\notin\{Y>0\}$.

(iii) If $\gs=\nu_{-}$ then $h(\gs)=\mu$, so $Y_{*}=0$ and $P_{*}=P_{-}$.
\end{proof}

\begin{lemma}\label{lem:rate}
Let $\nu\in\{\gs,\nu_{-},\nu_{+}\}$ and suppose that $x(t)\to\nu$ with
$\int^{\infty}|x(t)-\nu|\,dt<+\infty$. Then
$\lim_{r\to+\infty}r^{\nu}v(r)$ exists in $(0,+\infty)$, and equals
$\big(\lim_{t\to\infty}Y(t)\big)^{1/c}$ when $\nu=\gs$. Symmetrically, if the
orbit is defined for all $t\le t_{1}$ and $x(t)\to\nu$ as $t\to-\infty$ with
$\int_{-\infty}^{t_{1}}|x(t)-\nu|\,dt<+\infty$, then
$\lim_{r\to0^{+}}r^{\nu}v(r)$ exists in $(0,+\infty)$.
\end{lemma}

\begin{proof}
Write $\tilde v(t):=v(\mathrm e^{t})$, so that
$\frac{d}{dt}\ln\tilde v=\frac{rv'}{v}=-x$ and
\[
\ln\big(r^{\nu}v(r)\big)=\ln\tilde v(t)+\nu t
=\ln\tilde v(t_{1})+\nu t_{1}-\int_{t_{1}}^{t}\big(x(\tau)-\nu\big)\,d\tau ,
\]
and the integral converges as $t\to+\infty$ by assumption; exponentiating gives
the first claim. If $\nu=\gs$, then $r^{\gs}v=Y^{1/c}$ by \eqref{eq:recover}.
The backward statement follows from the same identity, letting $t\to-\infty$.
\end{proof}

\begin{proof}[Proof of Theorem~\ref{thm:main}]
By Theorem~\ref{thm:dulac} the orbit converges to $P_{*}$, $P_{-}$ or $P_{+}$.

\smallskip
\noindent\emph{Limit $P_{*}$ with $h(\gs)>\mu$ (case (I)).} By
Lemma~\ref{lem:linear}(ii), $P_{*}$ is a hyperbolic equilibrium whose two
eigenvalues have negative real part. Any orbit converging to such an
equilibrium does so exponentially. Indeed, let $A$ be the Jacobian at $P_{*}$
and choose $\sigma>0$ with $\max\{\operatorname{Re}\lambda:\ \lambda\in
\operatorname{spec}A\}<-2\sigma$; by the Lyapunov equation, $\R^{2}$ carries an
inner product $\langle\cdot,\cdot\rangle_{*}$, with associated norm
$|\cdot|_{*}$, such that $\langle Az,z\rangle_{*}\le-2\sigma|z|_{*}^{2}$ for
every $z$. Since $\mathbf F(P_{*}+z)=Az+o(|z|)$ as $z\to0$, there is $\rho>0$
with $\langle\mathbf F(P_{*}+z),z\rangle_{*}\le-\sigma|z|_{*}^{2}$ whenever
$|z|_{*}\le\rho$; writing $z(t)=\big(x(t),Y(t)\big)-P_{*}$, which lies in that
ball for $t$ large, we obtain
$\frac{d}{dt}|z|_{*}^{2}\le-2\sigma|z|_{*}^{2}$. Hence there are
$C,\sigma>0$ with
\begin{equation}\label{eq:expdecay}
|x(t)-\gs|+|Y(t)-Y_{*}|\le C\,e^{-\sigma t}.
\end{equation}
In particular $\int^{\infty}|x-\gs|\,dt<\infty$ and $Y(t)\to Y_{*}$, so
Lemma~\ref{lem:rate} gives
$r^{\gs}v(r)\to Y_{*}^{1/c}=\big(h(\gs)-\mu\big)^{1/(q-p+1)}=A_{*}$, and
$\frac{rv'}{v}=-x\to-\gs$ exponentially.

\smallskip
\noindent\emph{Limit $P_{-}$ with $\gs<\nu_{-}$ (case (II)).} By
Lemma~\ref{lem:linear}(i) both eigenvalues
$\lambda_{1}^{-}=-(N-p-p\nu_{-})$ and $\lambda_{2}^{-}=c(\gs-\nu_{-})$ are
negative, so $P_{-}$ is hyperbolic and asymptotically stable; the argument just
given yields \eqref{eq:expdecay} with $(\gs,Y_{*})$ replaced by
$(\nu_{-},0)$, and Lemma~\ref{lem:rate} yields
$r^{\nu_{-}}v(r)\to C\in(0,+\infty)$ and $\frac{rv'}{v}\to-\nu_{-}$.

\smallskip
\noindent\emph{Limit $P_{+}$ (case (III)).} By Lemma~\ref{lem:linear}(i),
$P_{+}$ is a hyperbolic saddle, with $\lambda_{1}^{+}>0>\lambda_{2}^{+}$. An
orbit converging to a hyperbolic equilibrium lies, for large $t$, on its local
stable manifold \cite[Ch.~IX]{Hartman}, along which convergence is exponential
at any rate $\sigma<|\lambda_{2}^{+}|$. Hence \eqref{eq:expdecay} holds with
$(\nu_{+},0)$ in place of $(\gs,Y_{*})$, and Lemma~\ref{lem:rate} gives
$r^{\nu_{+}}v(r)\to C\in(0,+\infty)$ and $\frac{rv'}{v}\to-\nu_{+}$.

\smallskip
\noindent\emph{Limit $(\gs,0)$ with $\gs=\nu_{-}$ (case (IV)).} Here
$h(\gs)=\mu$ and $P_{*}=P_{-}=(\gs,0)$. By Lemma~\ref{lem:linear}(i) the
Jacobian at $(\gs,0)$ has eigenvalues $\lambda_{1}=-a<0$ and $\lambda_{2}=0$,
where
\[
a:=N-p-p\gs>0,\qquad \ell_{0}:=\frac{\gs^{2-p}}{p-1}>0,\qquad c=q-p+1 .
\]
Writing $\xi=x-\gs$, $y=Y$, the system reads
\[
\dot\xi=-a\xi+\ell_{0}\,y+R(\xi,y),\qquad \dot y=-c\,\xi\,y ,
\]
with $R$ analytic and $R(\xi,y)=O(\xi^{2}+y^{2})$ (Remark~\ref{rem:analytic}).
By the centre manifold theorem \cite{Carr} there is a local invariant manifold
$\xi=\psi(y)$, with $\psi(y)=\frac{\ell_{0}}{a}y+O(y^{2})$, on which the flow is
governed by the reduced equation
\begin{equation}\label{eq:reduced}
\dot y=-c\,\psi(y)\,y=-\frac{c\,\ell_{0}}{a}\,y^{2}+O(y^{3}) .
\end{equation}
The linearisation at $(\gs,0)$ has no unstable direction, the half-plane
$\{y\ge0\}$ is invariant (because $\{Y=0\}$ is), and $y=0$ is a stable
equilibrium of \eqref{eq:reduced} restricted to $y\ge0$. The asymptotic phase
property of the centre manifold \cite{Carr} therefore applies on that
half-plane: there are a solution $\tilde y>0$ of \eqref{eq:reduced} and
$\sigma>0$ with
\[
y(t)=\tilde y(t)+O\big(\mathrm e^{-\sigma t}\big),
\qquad
\xi(t)=\psi\big(\tilde y(t)\big)+O\big(\mathrm e^{-\sigma t}\big)
\qquad\text{as }t\to+\infty .
\]
Integrating \eqref{eq:reduced} gives
$\tilde y(t)=\frac{a}{c\,\ell_{0}\,t}\big(1+o(1)\big)$; the exponentially small
corrections are then negligible, so that
$y(t)=\frac{a}{c\,\ell_{0}\,t}\big(1+o(1)\big)$ and
$\xi(t)=\frac{\ell_{0}}{a}y(t)+O\big(y(t)^{2}\big)=\frac{1}{ct}+O(t^{-2})$. By
\eqref{eq:recover}, $v(r)=Y^{1/c}r^{-\gs}$, so
\[
v(r)=\Big(\frac{a}{c\ell_{0}}\Big)^{1/c}r^{-\gs}(\ln r)^{-1/c}\big(1+o(1)\big),
\qquad
\frac{rv'(r)}{v(r)}=-\gs-\frac{1}{c\ln r}+O\big((\ln r)^{-2}\big).
\]
Finally
\[
\frac{a}{c\ell_{0}}=\frac{(p-1)\,(N-p-p\gs)\,\gs^{p-2}}{c}
=\frac{(p-1)\big[(q-p+1)(N-p)-p(p+s)\big](p+s)^{p-2}}{(q-p+1)^{p}} ,
\]
which gives $C_{*}$.

\smallskip
\noindent\emph{Exclusiveness and occurrence.} That exactly one alternative
holds follows from Lemma~\ref{lem:which} together with
$\nu_{-}<\nu_{+}$ and $\gs\ne\nu_{+}$. All four occur, for suitable values of
the parameters: (I) is realised by the
explicit solution $v(r)=A_{*}r^{-\gs}$ when $h(\gs)>\mu$; (III) is realised
along the one-dimensional stable manifold of the saddle $P_{+}$, which by
Lemma~\ref{lem:linear}(i) is transversal to $\{Y=0\}$ (its eigenvector for
$\lambda_{2}^{+}$ is
$\big(\tfrac{\nu_{+}^{2-p}}{p-1},\,\lambda_{2}^{+}-\lambda_{1}^{+}\big)$ with
$\lambda_{2}^{+}-\lambda_{1}^{+}\ne0$) and therefore meets $\{Y>0\}$; (II) is
realised for $\gs<\nu_{-}$, where $P_{-}$ is an asymptotically stable node, so
that a whole neighbourhood of $P_{-}$ in $\{Y>0\}$ is attracted to it; and (IV)
is realised for $\gs=\nu_{-}$ along the centre manifold, which meets $\{y>0\}$
since $\psi(y)$ is defined for small $y>0$.
\end{proof}

\begin{proof}[Proof of Corollary~\ref{cor:H}]
Assume $\gs>\nu_{-}$. By Theorem~\ref{thm:main} and Lemma~\ref{lem:which}(i),
only alternatives (I) and (III) are possible. In case (I),
$r^{\gs}v\to A_{*}>0$; in case (III), $r^{\gs}v=r^{\gs-\nu_{+}}\cdot r^{\nu_{+}}v
\to0$ because $\gs<\frac{N-p}{p}<\nu_{+}$. This proves the stated equivalence.

Assume now that case (III) holds, so that there are $C_{1},C_{2},C_{3}>0$ and
$R_{1}>R_{0}$ with
\[
C_{1}r^{-\nu_{+}}\le v(r)\le C_{2}r^{-\nu_{+}},
\qquad
|v'(r)|=\frac{x(\ln r)}{r}\,v(r)\le C_{3}r^{-\nu_{+}-1},
\qquad r\ge R_{1} ,
\]
the last bound because $x(\ln r)\to\nu_{+}$, hence is bounded. Hence, for $r\ge R_{1}$,
\[
r^{N-1}|v'|^{p}\le C r^{N-1-p(\nu_{+}+1)},\qquad
r^{N-1-p}v^{p}\le Cr^{N-1-p-p\nu_{+}},
\]
\[
r^{N-1+s}v^{q+1}\le Cr^{N-1+s-\nu_{+}(q+1)} .
\]
The first two exponents are $<-1$ if and only if $N-p<p\,\nu_{+}$, which holds
by \eqref{eq:nupm}. For the third, \eqref{eq:qsuper} gives
$q+1>p^{*}_{s}=\frac{p(N+s)}{N-p}$, whence
$\nu_{+}(q+1)>\frac{N-p}{p}\cdot\frac{p(N+s)}{N-p}=N+s$, so that exponent is
$<-1$ as well. All three integrals therefore converge.
\end{proof}

%%%%%%%%%%%%%%%%%%%%%%%%%%%%%%%%%%%%%%%%%%%%%%%%%%%%%%%%%%%%%%%%%%%%%%%%%%%%%%%
\section{The critical regime: proof of Theorem~\ref{thm:critical}}
\label{sec:critical}
%%%%%%%%%%%%%%%%%%%%%%%%%%%%%%%%%%%%%%%%%%%%%%%%%%%%%%%%%%%%%%%%%%%%%%%%%%%%%%%

Throughout this section $q=p^{*}_{s}-1$, so $\gs=\frac{N-p}{p}$,
$Y_{*}=\mub-\mu>0$ and $\mathcal H$ is the first integral
\eqref{eq:firstintegral}. Recall from \eqref{eq:B} that
\[
b+1=\frac{p}{c}=\frac{N-p}{p+s},
\qquad
b+2=\frac{N+s}{p+s},
\qquad\text{so that}\qquad
\frac{b+2}{b+1}=\frac{N+s}{N-p} .
\]
We first describe the level sets of $\mathcal H$.

\begin{lemma}\label{lem:levelsets}
Let $\mathcal H$ be as in \eqref{eq:firstintegral} and put
\[
D:=\Big\{(x,Y):\ \nu_{-}<x<\nu_{+},\ \ 0<Y<\tfrac{b+2}{b+1}\big(h(x)-\mu\big)\Big\}.
\]
Then, in $\{x>0,\ Y>0\}$:
\begin{enumerate}
\item[(i)] $\{\mathcal H<0\}=D$, $D$ is bounded, and $P_{*}\in D$;
\item[(ii)] $\{\mathcal H=0\}$ is the curve
$\Gamma_{0}:=\big\{\big(x,\tfrac{b+2}{b+1}(h(x)-\mu)\big):\ \nu_{-}<x<\nu_{+}\big\}$,
which is a single orbit running from $P_{-}$ to $P_{+}$;
\item[(iii)] $P_{*}$ is the only critical point of $\mathcal H$ in
$\{x>0,\ Y>0\}$; it is a nondegenerate minimum and in fact the global minimum of
$\mathcal H$ there, with
$\mathcal H(P_{*})=-\frac{Y_{*}^{\,b+2}}{(p-1)(b+1)(b+2)}<0$;
\item[(iv)] for every $\kappa>0$ and every $x>0$ there is exactly one $Y>0$ with
$\mathcal H(x,Y)=\kappa$, and it satisfies $\mu+Y-h(x)>0$.
\end{enumerate}
\end{lemma}

\begin{proof}
Write $\mathcal H(x,Y)=\frac{Y^{b+1}}{p-1}\Lambda(x,Y)$ with
$\Lambda(x,Y)=\frac{\mu-h(x)}{b+1}+\frac{Y}{b+2}$; since $Y^{b+1}>0$, the sign
of $\mathcal H$ is that of $\Lambda$, and $\Lambda<0$ (resp.\ $=0$) exactly when
$Y<\frac{b+2}{b+1}(h(x)-\mu)$ (resp.\ $=$), which requires $h(x)>\mu$, i.e.\
$x\in(\nu_{-},\nu_{+})$. This gives (i) and the description of
$\{\mathcal H=0\}$ in (ii); $D$ is bounded because $\nu_{\pm}$ are finite and
$h$ is bounded on $[\nu_{-},\nu_{+}]$, and $P_{*}=(\gs,Y_{*})\in D$ because
$\frac{b+2}{b+1}(h(\gs)-\mu)=\frac{b+2}{b+1}Y_{*}>Y_{*}$.

That $\Gamma_{0}$ is a single orbit follows from the fact that $\mathcal H$ is a
first integral: on $\Gamma_{0}$ one has
$\mu+Y-h(x)=\frac{b+2}{b+1}(h-\mu)-(h-\mu)=\frac{h(x)-\mu}{b+1}>0$, so
$\dot x>0$ there, and $\Gamma_{0}$ is a graph over $x\in(\nu_{-},\nu_{+})$ whose
endpoints are $P_{-}$ and $P_{+}$; hence it is a single orbit travelling from
$P_{-}$ to $P_{+}$.

(iii) Critical points: $\partial_{Y}\mathcal H=\frac{Y^{b}}{p-1}(\mu-h(x)+Y)=0$
gives $Y=h(x)-\mu$, and
$\partial_{x}\mathcal H=-\frac{h'(x)Y^{b+1}}{(p-1)(b+1)}=0$ gives $h'(x)=0$,
i.e.\ $x=\frac{N-p}{p}=\gs$; then $Y=h(\gs)-\mu=Y_{*}$. At $P_{*}$,
$\partial_{YY}\mathcal H=\frac{Y_{*}^{b}}{p-1}>0$,
$\partial_{xx}\mathcal H=-\frac{h''(\gs)Y_{*}^{b+1}}{(p-1)(b+1)}>0$ because
$h''(\gs)<0$ ($\gs$ is the strict maximum point of $h$), and
$\partial_{xY}\mathcal H=-\frac{h'(\gs)Y_{*}^{b}}{p-1}=0$. Hence the Hessian is
diagonal and positive definite, and the stated value follows from
$\mu-h(\gs)=-Y_{*}$. Finally $\mathcal H\ge0>\mathcal H(P_{*})$ outside $D$ by
(i), while on the compact set $\overline D$ the function $\mathcal H$ is
continuous, vanishes on $\partial D$ and has $P_{*}$ as its only interior
critical point; hence $\mathcal H(P_{*})=\min_{\{x>0,Y>0\}}\mathcal H$.

(iv) For fixed $x>0$, $Y\mapsto\mathcal H(x,Y)$ tends to $0$ as $Y\to0^{+}$ and
to $+\infty$ as $Y\to+\infty$, and $\partial_{Y}\mathcal H$ vanishes only at
$Y=h(x)-\mu$ (when positive), being negative before and positive after. Hence
$\mathcal H(x,\cdot)$ is either strictly increasing on $(0,\infty)$ (if
$h(x)\le\mu$) or decreasing then increasing (if $h(x)>\mu$); in both cases the
equation $\mathcal H(x,Y)=\kappa>0$ has exactly one root, and it lies in the
increasing branch, i.e.\ $Y>h(x)-\mu$, which is $\mu+Y-h(x)>0$.
\end{proof}

\begin{proof}[Proof of Theorem~\ref{thm:critical}]
Let $(x,Y)$ be the orbit of $v$, and $\kappa:=\mathcal H(x(t),Y(t))$, constant by
Theorem~\ref{thm:dulac}(ii). By Theorem~\ref{thm:apriori}, $x(t)<\nu_{+}$ and
$Y$ is bounded. Since $\mathcal H\ge\mathcal H(P_{*})$ on $\{x>0,\ Y>0\}$ by
Lemma~\ref{lem:levelsets}(iii), we already know that
$\kappa\ge\mathcal H(P_{*})$.

\smallskip
\noindent\emph{Step 1: $\kappa\le0$.} Suppose $\kappa>0$. By
Lemma~\ref{lem:levelsets}(iv) the orbit satisfies $\mu+Y-h(x)>0$ for all $t$,
hence $\dot x>0$ and $x$ is strictly increasing. Moreover, for each $x>0$ the
equation $\mathcal H(x,Y)=\kappa$ has a unique root $Y=\phi_{\kappa}(x)>0$,
which lies in
the branch where $\partial_{Y}\mathcal H>0$; by the implicit function theorem
$\phi_{\kappa}$ is continuous and positive on $(0,+\infty)$, and along the orbit
$Y(t)=\phi_{\kappa}(x(t))$. By Theorem~\ref{thm:apriori} the orbit satisfies
$\delta\le x(t)<\nu_{+}$ for some $\delta>0$, and the continuous function
$x\mapsto\frac{x^{2-p}}{p-1}\big(\mu+\phi_{\kappa}(x)-h(x)\big)$ is strictly
positive on the compact interval $[\delta,\nu_{+}]$, hence bounded below there
by some $c_{1}>0$. Therefore $\dot x\ge c_{1}$ for all $t$, so $x$ reaches
$\nu_{+}$ in finite time, contradicting Theorem~\ref{thm:apriori}.

\smallskip
\noindent\emph{Step 2: $\kappa=0$ gives alternative {\rm(III)}.} By
Lemma~\ref{lem:levelsets}(ii) the orbit is the heteroclinic orbit
$\Gamma_{0}$, so $x(t)\to\nu_{+}$, $Y(t)\to0$ as $t\to+\infty$ and
$x(t)\to\nu_{-}$, $Y(t)\to0$ as $t\to-\infty$. Both $P_{\pm}$ are hyperbolic
saddles by Lemma~\ref{lem:linear}(i) (here $\nu_{-}<\gs<\nu_{+}$), so the
convergence is exponential in $t$ at both ends; Lemma~\ref{lem:rate} applied
forward and backward gives
$r^{\nu_{+}}v(r)\to C\in(0,+\infty)$ as $r\to+\infty$ and
$r^{\nu_{-}}v(r)\to C'\in(0,+\infty)$ as $r\to0^{+}$. Since $\Gamma_{0}$ is
defined and bounded for all $t\in\R$, Lemma~\ref{lem:system} turns it into a
positive solution of \eqref{eq:ode} on $(0,+\infty)$, which by
Remark~\ref{rem:uniqueness} agrees with $v$ on $(R_{0},+\infty)$; thus $v$
extends to a positive solution on $(0,+\infty)$.

\smallskip
\noindent\emph{Step 3: $\mathcal H(P_{*})\le\kappa<0$ gives {\rm(I)} or
{\rm(II)}.} By Lemma~\ref{lem:levelsets}(i) the orbit lies in the bounded set
$D$, so it is bounded; by Poincar\'e--Bendixson \cite[\S1.7]{DLA} its
$\omega$-limit set is an equilibrium, a periodic orbit, or a set containing
equilibria (see the proof of Theorem~\ref{thm:dulac}). The $\omega$-limit set is
contained in $\overline D$ and $\mathcal H\equiv\kappa$ on it; since the
equilibria lying in $\overline D$, namely $P_{*}$ and $P_{\pm}$, have
$\mathcal H\in\{0,\mathcal H(P_{*})\}$, the case
$\kappa\in\big(\mathcal H(P_{*}),0\big)$ excludes equilibria from the
$\omega$-limit set altogether. Therefore the $\omega$-limit set is a periodic
orbit
$\Gamma\subset\{\mathcal H=\kappa\}$. By Lemma~\ref{lem:levelsets}(iii),
$\kappa$ is a regular value of $\mathcal H$ in $D$, so
$L_{\kappa}:=\{\mathcal H=\kappa\}\cap D$ is a one-dimensional embedded
submanifold of $\R^{2}$. Being compact and connected, $\Gamma$ is a connected
component of $L_{\kappa}$, and therefore admits an open neighbourhood
$\mathcal N\subset\R^{2}$ with $\mathcal N\cap L_{\kappa}=\Gamma$. Since the
orbit accumulates on $\Gamma$ it eventually enters $\mathcal N$; being contained
in $L_{\kappa}$, it then lies on $\Gamma$, and by invariance it coincides with
$\Gamma$. Thus the orbit is periodic, of some period $T>0$. Being defined for all
$t\in\R$, it yields, exactly as in Step~2, a positive solution of
\eqref{eq:ode} on $(0,+\infty)$ which agrees with $v$ on $(R_{0},+\infty)$; so
$v$ extends to $(0,+\infty)$. By \eqref{eq:recover},
$r^{\gs}v(r)=\big(Y(\ln r)\big)^{1/c}$ for every $r>0$, and the function
$\Psi(t):=Y(t)^{1/c}$ is positive, $T$-periodic and non-constant (otherwise the
orbit would be an equilibrium), which is {\rm(II)}.
Finally $\kappa=\mathcal H(P_{*})$ forces the orbit to be $\{P_{*}\}$, i.e.\
$x\equiv\gs$ and $Y\equiv Y_{*}$, that is $v(r)=A_{*}r^{-\gs}$ for all large
$r$, hence for every $r>R_{0}$ by Remark~\ref{rem:uniqueness}; this is
{\rm(I)}.

\smallskip
\noindent\emph{Occurrence.} {\rm(I)} is the explicit solution; {\rm(III)} is
$\Gamma_{0}$; and {\rm(II)} occurs for every
$\kappa\in(\mathcal H(P_{*}),0)$: indeed $Y\mapsto\mathcal H(\gs,Y)$ decreases
from $0$ (as $Y\to0^{+}$) to $\mathcal H(P_{*})$ at $Y=Y_{*}$, so there is a
unique $Y_{\kappa}\in(0,Y_{*})$ with $\mathcal H(\gs,Y_{\kappa})=\kappa$, and
the argument of Step~3 applied to the orbit through $(\gs,Y_{\kappa})$ shows
that this orbit is periodic. It is the only one at the level $\kappa$, so that
$\kappa$ does index the family in {\rm(II)}: along a periodic orbit in $D$ the
component $Y$ is not constant (otherwise $\dot Y\equiv0$ would force
$x\equiv\gs$, hence $Y\equiv Y_{*}$ and the orbit would be $\{P_{*}\}$), so
$Y$ attains a strict maximum and a strict minimum, at which $\dot Y=0$ and
therefore $x=\gs$; thus every periodic orbit at the level $\kappa$ meets
$\{x=\gs\}$ in at least two points. On the other hand
$\mathcal H(\gs,\cdot)$ takes the value $\kappa$ at exactly two points: besides
$Y_{\kappa}$, exactly one point of $\big(Y_{*},\tfrac{b+2}{b+1}Y_{*}\big)$,
because $\partial_{Y}\mathcal H(\gs,Y)=\frac{Y^{b}}{p-1}(Y-Y_{*})>0$ for
$Y>Y_{*}$, so that $\mathcal H(\gs,\cdot)$ increases there from
$\mathcal H(P_{*})$ to the value $0$, attained at
$Y=\frac{b+2}{b+1}Y_{*}$.
\end{proof}

%%%%%%%%%%%%%%%%%%%%%%%%%%%%%%%%%%%%%%%%%%%%%%%%%%%%%%%%%%%%%%%%%%%%%%%%%%%%%%%
\section{The subcritical and sub-Serrin regimes: proofs of
Theorems~\ref{thm:sub} and~\ref{thm:nonex}}\label{sec:sub}
%%%%%%%%%%%%%%%%%%%%%%%%%%%%%%%%%%%%%%%%%%%%%%%%%%%%%%%%%%%%%%%%%%%%%%%%%%%%%%%

\begin{proof}[Proof of Theorem~\ref{thm:sub}]
Here $\frac{N-p}{p}<\gs<\nu_{+}$, so in particular $q<p^{*}_{s}-1$ and
$\nu_{-}<\gs<\nu_{+}$, i.e.\ $P_{*}$ exists in $\{Y>0\}$. By
Theorem~\ref{thm:apriori} the orbit satisfies \eqref{eq:apriori}, and by
Theorem~\ref{thm:dulac}(i) it converges to one of $P_{*}$, $P_{-}$, $P_{+}$.

The limit $P_{-}$ is impossible: since $\gs>\nu_{-}$, Lemma~\ref{lem:linear}(i)
shows that $P_{-}$ is a saddle whose local stable manifold is contained in the
invariant line $\{Y=0\}$, which the orbit never meets.

The limit $P_{*}$ forces the orbit to be constant: by
Lemma~\ref{lem:linear}(ii) the Jacobian at $P_{*}$ has trace
$p\gs-(N-p)>0$ and positive determinant, so both eigenvalues have positive real
part and $P_{*}$ is a repeller. Applying to the reversed field $-\mathbf F$ the
argument used in case {\rm(I)} of the proof of Theorem~\ref{thm:main}, we find
$\rho,\sigma>0$ such that $\frac{d}{dt}|z|_{*}^{2}\ge2\sigma|z|_{*}^{2}$
whenever $0<|z|_{*}\le\rho$, where $z=(x,Y)-P_{*}$; hence every orbit other
than $\{P_{*}\}$ leaves the ball $\{|z|_{*}\le\rho\}$ in finite time and cannot
converge to $P_{*}$ as $t\to+\infty$. Hence $x\equiv\gs$, $Y\equiv Y_{*}$ and
$v(r)=A_{*}r^{-\gs}$ for all large $r$, therefore for every $r>R_{0}$ by
Remark~\ref{rem:uniqueness}. This gives alternative {\rm(I)}.

Finally, if the limit is $P_{+}$, then $P_{+}$ is a hyperbolic saddle
(Lemma~\ref{lem:linear}(i), using $\gs<\nu_{+}$), the orbit lies on its
one-dimensional local stable manifold, convergence is exponential, and
Lemma~\ref{lem:rate} gives $r^{\nu_{+}}v(r)\to C\in(0,+\infty)$ and
$rv'/v\to-\nu_{+}$; this is {\rm(III)}. That both alternatives occur is clear:
{\rm(I)} is the explicit solution, and the stable manifold of $P_{+}$ is
transversal to $\{Y=0\}$ (its eigenvector for $\lambda_{2}^{+}$ has a nonzero
second component because $\lambda_{2}^{+}\ne\lambda_{1}^{+}$), hence meets
$\{Y>0\}$.
\end{proof}

\begin{proof}[Proof of Theorem~\ref{thm:nonex}]
Here $\gs\ge\nu_{+}$, hence $q\le q_{S}<p^{*}_{s}-1$. Suppose, for
contradiction, that $v>0$ solves \eqref{eq:ode} on $(R_{0},+\infty)$. By
Theorem~\ref{thm:apriori} the orbit satisfies
\eqref{eq:apriori}, and by Theorem~\ref{thm:dulac}(i) it converges to one of the
equilibria \eqref{eq:equilibria}. We rule out all of them.

\emph{$P_{*}$ does not exist.} Indeed $\gs\ge\nu_{+}$ gives $h(\gs)\le
h(\nu_{+})=\mu$ by \eqref{eq:trichotomycases}, so $Y_{*}=h(\gs)-\mu\le0$.

\emph{$P_{-}$ is unreachable.} Since $\gs\ge\nu_{+}>\nu_{-}$,
Lemma~\ref{lem:linear}(i) shows that $P_{-}$ is a saddle whose local stable
manifold lies in $\{Y=0\}$.

\emph{$P_{+}$ is unreachable.} If $\gs>\nu_{+}$, then
$\lambda_{1}^{+}>0$ and $\lambda_{2}^{+}=c(\gs-\nu_{+})>0$ by
Lemma~\ref{lem:linear}(i), so $P_{+}$ is a repeller and no orbit with $Y>0$
converges to it. If $\gs=\nu_{+}$, then $\lambda_{1}^{+}>0$ and
$\lambda_{2}^{+}=0$. Writing $\xi=x-\nu_{+}$, $y=Y$, the system becomes
\[
\dot\xi=\lambda_{1}^{+}\xi+\ell_{0}\,y+O(\xi^{2}+y^{2}),
\qquad
\dot y=-c\,\xi\,y,
\qquad \ell_{0}:=\frac{\nu_{+}^{2-p}}{p-1}>0 .
\]
Since there is no stable direction, the local centre manifold theorem
\cite{Carr} (or the reduction principle, \cite[Ch.~IX]{Hartman}) shows that any
orbit remaining in a small neighbourhood of $P_{+}$ for all large $t$ lies on
the one-dimensional local centre manifold $\xi=\psi(y)$, with
$\psi(y)=-\frac{\ell_{0}}{\lambda_{1}^{+}}y+O(y^{2})$ (this value of $\psi'(0)$
follows from the invariance relation
$\lambda_{1}^{+}\psi(y)+\ell_{0}y+O(y^{2})=\psi'(y)\big(-c\,\psi(y)y\big)$).
On it the reduced flow is
\[
\dot y=-c\,\psi(y)\,y=\frac{c\,\ell_{0}}{\lambda_{1}^{+}}\,y^{2}+O(y^{3})>0
\qquad\text{for small }y>0 ,
\]
so $y$ is increasing and the orbit leaves the neighbourhood; hence no orbit with
$Y>0$ converges to $P_{+}$.

All three possibilities being excluded, we contradict
Theorem~\ref{thm:dulac}(i), which asserts that the orbit converges to one of the
equilibria \eqref{eq:equilibria}. Hence no such $v$ exists.
\end{proof}

\section*{Declarations}

%\noindent\textbf{Funding.} This research is funded by ...

\smallskip
\noindent\textbf{Conflict of interest.} The author declares that he has no
conflict of interest.

\smallskip
\noindent\textbf{Data availability.} Data sharing is not applicable to this
article, as no datasets were generated or analysed during the current study.

\smallskip
\noindent\textbf{Use of AI tools.} AI-assisted tools were used in developing and drafting this work; the author has verified its contents and takes full responsibility.

\end{document}